\documentclass[12pt]{amsart}
\usepackage[utf8]{inputenc}

\usepackage{amsrefs}
\usepackage{supertabular}
\usepackage{amsxtra,amssymb,amsmath,amscd,url,enumitem}

\setenumerate{itemsep=5pt}
\setitemize{itemsep=5pt}

\newlist{enumth}{enumerate}{1}
\setlist[enumth]{label=\emph{(\arabic*)}, ref=\thetheorem(\arabic*)}

\usepackage[utf8]{inputenc}
\usepackage[english]{babel}
\usepackage{fullpage}
\usepackage{scrtime}
\usepackage[colorlinks]{hyperref}
\usepackage{datetime}
\usepackage[all]{xy}
\usepackage{tikz-cd}
\usepackage{mathrsfs}
\usepackage{xspace}

\usepackage{tensor}
\usepackage{braket}

\usepackage{todonotes}
\usepackage[capitalize]{cleveref}
\crefalias{enumthi}{theorem}
\crefname{section}{\S}{\SS}
\crefname{figure}{Figure}{Figures}

\shortdate

\DeclareMathSymbol{A}{\mathalpha}{operators}{`A}%
\DeclareMathSymbol{B}{\mathalpha}{operators}{`B}%
\DeclareMathSymbol{C}{\mathalpha}{operators}{`C}%
\DeclareMathSymbol{D}{\mathalpha}{operators}{`D}%
\DeclareMathSymbol{E}{\mathalpha}{operators}{`E}%
\DeclareMathSymbol{F}{\mathalpha}{operators}{`F}%
\DeclareMathSymbol{G}{\mathalpha}{operators}{`G}%
\DeclareMathSymbol{H}{\mathalpha}{operators}{`H}%
\DeclareMathSymbol{I}{\mathalpha}{operators}{`I}%
\DeclareMathSymbol{J}{\mathalpha}{operators}{`J}%
\DeclareMathSymbol{K}{\mathalpha}{operators}{`K}%
\DeclareMathSymbol{L}{\mathalpha}{operators}{`L}%
\DeclareMathSymbol{M}{\mathalpha}{operators}{`M}%
\DeclareMathSymbol{N}{\mathalpha}{operators}{`N}%
\DeclareMathSymbol{O}{\mathalpha}{operators}{`O}%
\DeclareMathSymbol{P}{\mathalpha}{operators}{`P}%
\DeclareMathSymbol{Q}{\mathalpha}{operators}{`Q}%
\DeclareMathSymbol{R}{\mathalpha}{operators}{`R}%
\DeclareMathSymbol{S}{\mathalpha}{operators}{`S}%
\DeclareMathSymbol{T}{\mathalpha}{operators}{`T}%
\DeclareMathSymbol{U}{\mathalpha}{operators}{`U}%
\DeclareMathSymbol{V}{\mathalpha}{operators}{`V}%
\DeclareMathSymbol{W}{\mathalpha}{operators}{`W}%
\DeclareMathSymbol{X}{\mathalpha}{operators}{`X}%
\DeclareMathSymbol{Y}{\mathalpha}{operators}{`Y}%
\DeclareMathSymbol{Z}{\mathalpha}{operators}{`Z}%

\newcommand{\respup}[1]{\emph{(}\resp {#1}\emph{)}}

\renewcommand{\leq}{\leqslant}
\renewcommand{\geq}{\geqslant}
 
\numberwithin{equation}{section}

\def\setminus{\mathchoice
    {\mathbin{\vrule height .62ex width 1.61ex depth -.38ex}}
    {\mathbin{\vrule height .62ex width 1.61ex depth -.38ex}}
    {\mathbin{\vrule height .50ex width 0.85ex depth -.28ex}}
    {\mathbin{\vrule height .20ex width 0.570ex depth -.24ex}}
}

\renewcommand{\mathcal}{\mathscr}

\newcommand{\Cc}{\mathbf{C}}

\newcommand{\Zz}{\mathbf{Z}}

\newcommand{\Rr}{\mathbf{R}}
\newcommand{\Gg}{\mathbf{G}}

\newcommand{\Qq}{\mathbf{Q}}

\newcommand{\Ff}{\mathbf{F}}

\newcommand{\bQl}{\overline{\Qq}_{\ell}}

\def\loccit{loc.\kern3pt cit.{}\xspace}
\def\cf{see\kern.3em}
\def\Cf{See\kern.3em}
\def\eg{e.g.\kern.3em}

\def\resp{\text{resp.}\kern.3em}

\newcommand{\mods}[1]{\,(\mathrm{mod}\,{#1})}

\newcommand{\mfm}{\mathfrak{m}}

\newcommand{\lra}{\longrightarrow}

\newcommand{\fleche}[1]{\stackrel{#1}{\lra}}

\DeclareMathOperator{\ord}{ord}

\DeclareMathOperator{\Div}{Div}

\DeclareMathOperator{\divv}{div}

\DeclareMathOperator{\Tr}{Tr}

\DeclareMathOperator{\swan}{swan}

\newcommand{\eps}{\varepsilon}
\renewcommand{\rho}{\varrho}

\DeclareMathOperator{\SL}{\mathbf{SL}}
\DeclareMathOperator{\GL}{\mathbf{GL}}

\DeclareMathOperator{\SO}{\mathbf{SO}}

\DeclareMathOperator{\SU}{\mathbf{SU}}
\DeclareMathOperator{\Un}{\mathbf{U}}

\newcommand{\demi}{{\textstyle{\frac{1}{2}}}}

\DeclareMathSymbol{\gena}{\mathord}{letters}{"3C}
\DeclareMathSymbol{\genb}{\mathord}{letters}{"3E}

\theoremstyle{plain}
\newtheorem{theorem}{Theorem}[section]
\newtheorem*{theorem*}{Theorem}
\newtheorem{lemma}[theorem]{Lemma}

\newtheorem{proposition}[theorem]{Proposition}

\theoremstyle{remark}

\theoremstyle{definition}
\newtheorem{question}[theorem]{Question}

\newtheorem{definition}[theorem]{Definition}

\newtheorem{remark}[theorem]{Remark}

\AddToHook{env/lemma/begin}{\crefalias{theorem}{lemma}}
\AddToHook{env/proposition/begin}{\crefalias{theorem}{proposition}}
\AddToHook{env/corollary/begin}{\crefalias{theorem}{corollary}}
\AddToHook{env/conjecture/begin}{\crefalias{theorem}{conjecture}}
\AddToHook{env/claim/begin}{\crefalias{theorem}{claim}}
\AddToHook{env/remark/begin}{\crefalias{theorem}{remark}}

\newcommand{\abs}[1]{\left\lvert#1\right\rvert}

\newcommand{\mcL}{\mathscr{L}}

\newcommand{\mcX}{\mathscr{X}}

\renewcommand{\geq}{\geqslant}
\renewcommand{\leq}{\leqslant}

\begin{document}

\title{Jacobian graphs}

\author{Arthur Forey}
\address[A. Forey]{Univ. Lille, CNRS, UMR 8524 - Laboratoire Paul Painlevé,  \newline F-59000 Lille, France} 
  \email{arthur.forey@univ-lille.fr}

\author{Javier Fresán}
\address[J. Fres\'an]{Sorbonne Université and Université Paris Cité, CNRS, IMJ-PRG, \newline F-75005 Paris, France}
\email{javier.fresan@imj-prg.fr}

\author{Emmanuel Kowalski}
\address[E. Kowalski]{D-MATH, ETH Zürich, Rämistrasse 101, 8092 Zürich, Switzerland} 
\email{kowalski@math.ethz.ch}

\author{Yuval Wigderson}
\address[Y. Wigderson]{ETH-ITS, Scheuchzerstrasse 70, 8006 Zürich, Switzerland} 
\email{yuval.wigderson@eth-its.ethz.ch}

\begin{abstract}
  We introduce \emph{jacobian graphs}, which are explicit
  families of regular graphs that are spectrally indistinguishable from
  random graphs, but whose local structure is very different from that
  of random graphs. The construction relies on the geometric properties
  of generalized jacobians of curves and on general equidistribution
  theorems for character sums over finite fields.
\end{abstract}




\maketitle 

\setcounter{tocdepth}{1}
\tableofcontents

\section{Introduction}

In the companion paper~\cite{short}, we constructed explicit families of
graphs which have very strong spectral pseudorandomness properties, and
yet differ strikingly from random graphs by the fact that they do not contain
either $K_{2,3}$ or $C_4$ as subgraphs. These graphs are parameterized by a finite field~$k$, and their properties relied on two facts: (1) that
certain explicit subsets of $k\times k$ are Sidon sets (or close relatives
of Sidon sets); (2) that certain deep equidistribution theorems due to
Katz and Deligne for families of exponential sums over finite fields
imply that the empirical spectral distribution is close to the
semicircle distribution. These are rather remarkable properties: as we discuss in much greater detail in \cite{short}, the latter property implies that our graphs are ``spectrally indistinguishable'' from random graphs of the same density, whereas the former property is strikingly different from the behavior of random graphs of the same density. 

The starting point of the current paper is the fact that both (1) and
(2) can be generalized considerably.  This allows us in this paper to
construct many new examples of graphs with properties similar to those
defined in~\cite{short}.  On the other hand, in doing so we have to rely
on more advanced techniques from algebraic and arithmetic
geometry. This means that the resulting graphs, which we call
\emph{jacobian graphs}, are rather less transparent to the broader
combinatorics community. However, we view these as interesting test
cases for many extremal combinatorial problems, in view of the richness
of the underlying geometry.

Our goal in this paper is not ``merely'' to generalize the construction in \cite{short}. The constructions we gave in \cite{short} are rather special: we were only able to construct graphs on $p^{2k}$ vertices, where $p$ is a prime, and gave at most two non-isomorphic graphs of that order. By contrast, the constructions in the present paper are much more varied, in a number of ways: we can construct a graph satisfying (1) and (2) on nearly every number $n$ of vertices (see Section~\ref{sec-graphs}), and moreover the number of ``genuinely different constructions'' is rather large. More precisely, our graphs are defined in terms of curves of genus $1$ or $2$ over finite fields, and these are parameterized by moduli spaces of dimension $1$ and $3$, respectively. This means that, rather than arising as ``sporadic'' examples as in \cite{short}, the construction of jacobian graphs in the present paper is much more ``robust''. We are hopeful that this extra flexibility will manifest itself as further applications; concretely, it may be possible to prove the existence of a jacobian graph with additional desirable properties without explicitly constructing it, by instead averaging over the large family of different jacobian graphs (see \cref{rm-moduli}). 

This paper will mostly be written with an audience familiar with the
language of algebraic geometry, although we provide an appendix with
concrete descriptions of some of the basic objects (and we emphasize
that these are indeed concrete enough to allow us to compute fully
explicitly the graphs that we describe).

A jacobian graph is associated to the following data:
\begin{itemize}
\item a field~$k$; 
\item a smooth projective geometrically connected algebraic curve~$C$
  defined over~$k$; 
\item a $k$-rational effective divisor $\mfm$ on~$C$, called a
  \emph{modulus}; 
\item a $k$-rational divisor $\delta$ of degree~$1$ on~$C$.
\end{itemize}

Given these data, one can define:
\begin{itemize}
\item the generalized jacobian $J_{\mfm}$, which is a connected algebraic group defined over~$k$, of dimension $\dim(J_{\mfm})=g+\max(\deg(\mfm)-1,0)$, where $g$ is the genus of $C$; 
\item an Abel--Jacobi map
$i_{\delta}\colon C\setminus \mfm\to J_{\mfm}$, which  is a closed
immersion if $\dim(J_{\mfm}) \geq 2$.  
\end{itemize}

\begin{definition} The corresponding \emph{jacobian graph}\footnote{\,We omit~$\delta$ from the notation.} $\Gamma(C,\mfm;k)$ is the
graph\footnote{\,In this paper, graphs are undirected, without multiple
  edges, but may have loops.} with vertex set~$J_{\mfm}(k)$ and with
an edge joining~$x$ to~$y$ if and only if $x+y\in i_{\delta}(C\setminus\mfm)(k)$. 
\end{definition}

We will be particularly interested in the case where the dimension
of~$J_{\mfm}$ is~$2$ and $g\geq 1$, which implies that~$g=1$ or~$2$.  In
this case, it was proved in~\cite{ffk2} that the image~$S(k)$
of the~$k$\nobreakdash-points of~$C\setminus \mfm$ in~$J_{\mfm}(k)$ is a \emph{symmetric Sidon set},
i.e., that there exists an element $a_0$ of~$J_{\mfm}(k)$, called the
\emph{center}, such that $S(k)=a_0-S(k)$, and furthermore the equation
\[
  a+b=c+d
\]
with $(a,b,c,d)$ in~$S(k)$ has only the solutions where $a\in \{c,d\}$
and where $b=a_0-a$.

\begin{theorem}\label{th-jacobian}
  With notation as above, assume that~$k$ is a finite field. For
  $n\geq 1$, denote by $k_n$ the extension of degree~$n$ of~$k$ in a
  fixed algebraic closure~$\bar{k}$ of~$k$. Denote further
  \[
  \Gamma_n(C,\mfm)=\Gamma(C,\mfm;k_n).
 \]  Assume that $\dim(J_{\mfm})=2$
  and the genus of~$C$ is either~$1$ or~$2$.
  \begin{enumth}
  \item The graph $\Gamma_n(C,\mfm)$ is a finite graph, $d_n$-regular
    with $d_n=|i_{\delta}(C\setminus \mfm)(k_n)|$, and contains no
    subgraph
    $K_{2,3}=\tikz[every node/.style={fill, circle, inner
      sep=1.2pt},xscale=.35, yscale=.25,baseline=2.1ex]{\foreach \i in
      {1,2} \node (a\i) at (0,\i+.5) {}; \foreach \j in {1,2,3} \node
      (b\j) at (1,\j) {}; \foreach \i in {1,2} {\foreach \j in {1,2,3}
        \draw (a\i) -- (b\j);}}$.

  \item\label{it-ram1} The non-trivial spectrum\footnote{\,The spectrum
      refers to that of the Markov averaging operator of the graph,
      i.e., to that of $d_n^{-1}A$, where~$A$ is the adjacency
      matrix. The eigenvalue~$1$ is called the \emph{trivial}
      eigenvalue.} of~$\Gamma_n(C,\mfm)$ is contained in an interval
    \[
      \Bigl[-\frac{2}{\sqrt{|k_n|}}+O(|k_n|^{-1}),
      \frac{2}{\sqrt{|k_n|}}+O(|k_n|^{-1})
      \Bigr]
    \]
    for all~$n\geq 1$.

  \item\label{it-ram2} If~$C$ is \emph{ordinary}, then there exists a
    sequence of values of~$n$ of positive density such that~$\Gamma_n(C,\mfm)$ is a Ramanujan graph, i.e., such that the
    non-trivial spectrum of~$\Gamma_n(C,\mfm)$ is contained in the
    interval
    \[
      \Bigl[-\frac{2\sqrt{d_n-1}}{d_n}, \frac{2\sqrt{d_n-1}}{d_n}
      \Bigr].
    \]

  \item\label{it-equi} Suppose that the center of the symmetric Sidon
    set $i_{\delta}((C\setminus \mfm)(k_n))$ belongs to $2J_{\mfm}(k)$
    and that one of the following conditions holds: \smallskip
    \begin{enumerate}
    \item[\textup{(i)}] $g=2$.
    \item[\textup{(ii)}] $g=1$, $\mfm$ is a double point and the
      integer~$a_C(k)$ defined by $|C(k)|=|k|+1-a_C(k)$ is either~$0$ or has
      multiplicative order at least~$5$ in $k^{\times}$.
    \item[\textup{(iii)}] $g=1$ and, for some prime $\ell$ invertible
      in~$k$, the set of unramified characters for the object
      $i_{\delta,*}\bQl[1](1/2)$ on~$J_{\mfm}$, in the sense of
      Forey--Fresán--Kowalski~\cite{ffk}, coincides with the set of
      characters which are trivial on the kernel of the natural
      projection $J_{\mfm}\to C$.
    \end{enumerate} Then, as~$n\to+\infty$, the numbers
    \[
      \lambda \frac{d_n}{\sqrt{d_n-1}},
    \]
    where $\lambda$ ranges over the non-trivial eigenvalues
    of~$\Gamma_n(C,\mfm)$, counted with multiplicity, become
    equidistributed according to the semicircle
    law
    \[
      \mu_{\mathrm{sc}}=\frac{1}{\pi}\sqrt{1-\frac{x^2}{4}}dx
    \]
    supported on the interval~$[-2,2]$.
 
   \item More precisely, for $n\geq 1$, the estimate
    \begin{equation}\label{it-w1}
      W_1\Bigl( \frac{1}{|J_{\mfm}(k_n)|-1} \sum_{\substack{\lambda\in
        \mathrm{Sp}(\Gamma_n(C,\mfm))\\\lambda\not=1}} \delta_{\lambda
      d_n/\sqrt{d_n-1}},
      \mu_{\mathrm{sc}}
      \Bigr)\ll \frac{1}{n},
    \end{equation}
    holds, where the sum over the spectrum is performed with
    multiplicity and $W_1$ denotes the $1$-Wasserstein distance between
    measures on~$\Rr$.
  \end{enumth}
\end{theorem}

These properties are similar to those of the graphs described
in~\cite{short}, and we refer the reader to the introduction and discussion
in~\cite{short} for comments on the context and combinatorial
background, as well as motivation for such constructions.

These properties suggest that it would be interesting to understand
further invariants of jacobian graphs. In fact, the answer to the
following question could have remarkable consequences in Ramsey theory, as we discuss in \cite{short}*{\S\,2.3}. 

\begin{question}\label{qu:ind}
  What is the independence number\footnote{\,Recall that this is the
    largest order of a subset~$I$ of vertices such that no $x\neq y$
    in~$I$ are joined by an edge. If the graph contains loops, we
    ignore them when considering independent sets.} of jacobian graphs?
  In particular, does there exist $\delta>0$ such that the independence
  number of a jacobian graph with $n$ vertices is~$O(n^{3/4-\delta})$? Even if this is not the case for all jacobian graphs, is it true for some infinite sequence of jacobian graphs?
\end{question}

\begin{remark}\label{rm-intro}
  Here are some remarks concerning~\cref{th-jacobian}.
  
  \begin{enumerate}
  \item The proof of this theorem relies on deep tools of algebraic
    geometry; besides our earlier construction in~\cite{short}, a useful
    and fair\footnote{\,However, we stress that in our setting, the
      degree tends to infinity with the order of the graph. In this
      regime, it is much easier to construct Ramanujan graphs, and
      proving that certain explicit sequences of graphs (e.g.\ Paley
      graphs) are Ramanujan can be done using purely elementary
      arguments.} comparison is the construction of the original
    Ramanujan graphs \cite{lps}, which also depends essentially on the
    Riemann hypothesis over finite fields (as well as on the relation
    between Hecke eigenvalues of modular forms and varieties over finite
    fields, due to Eichler, Igusa and Shimura).
  
    \item The jacobian graphs are ``explicit'', and can be efficiently
    constructed once the data that define them is given
    (Figure~\ref{fig:spectrum data} below demonstrates this by giving
    the numerical data in one more or less random example). On the other
    hand, it is not necessarily straightforward to construct such a
    graph with $n$ vertices for a given integer~$n$. However, it is easy
    to find one with~$(1+o(1))n$ vertices, at least using a
    probabilistic algorithm, even if one wishes to have a true Ramanujan
    graph (see Proposition~\ref{pr-n}).
    
    \begin{figure}[ht]
      \includegraphics[width=.45\textwidth]{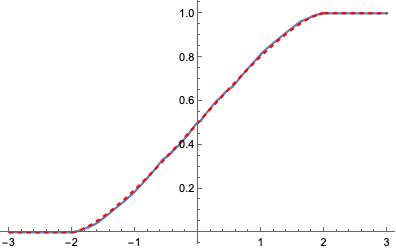}
      \includegraphics[width=.45\textwidth]{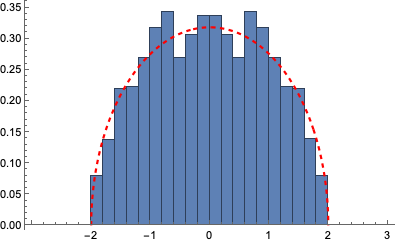}
      \caption{Numerical data for the spectrum of the jacobian graph
        associated to the genus~$2$ curve defined by $y^2=x^5-x^4+2x^3+2$
        over $\Ff_{53}$ and the trivial modulus. The graph has $2660$
        vertices. In both charts, the blue data are the spectrum of the
        graph, and the red curve is the semicircle distribution; the first
        chart shows the cumulative distribution function, and the second
        shows a histogram of the probability density function.}
      \label{fig:spectrum data}
    \end{figure}
    
  \item It is noteworthy that the general jacobian graphs vary in
    ``continuous'' families, as explained in
    Remark~\ref{rm-moduli}. This is an unusual feature for ``extremal''
    constructions based on algebraic or geometric objects, including
    those from~\cite{short}, and we are intrigued by the possibility of
    exploiting these degrees of freedom in some applications (see
    Remark~\ref{rm-moduli}).

\item Another contrasting feature is that, whereas all non-trivial
    eigenvalues of the graphs in~\cite{short} have large multiplicity
    (about the square root of the number of vertices), we expect that very
    few eigenvalues of jacobian graphs should have any
    multiplicity. This is confirmed experimentally. See \cite{short}*{\S\,2.1} for more details about eigenvalue multiplicity and its relation to the phenomenon of eigenvalue repulsion in the study of random matrices.

  \item The condition that $\Gamma_n(C,\mfm)$ is a Ramanujan graph in
    Theorem\,\ref{it-ram2} means in particular that it enjoys
    essentially optimal spectral pseudorandomness properties; see, e.g.,
    the survey of Krivelevich and Sudakov~\cite{ks}*{page\,19}.

  \item For comparison, the graphs constructed in~\cite{short} arise in
    a similar way, but with the underlying group $J_{\mfm}(k_n)$
    replaced with the additive group $k_n\times k_n$ and the image of
    the Abel--Jacobi map replaced by either the graph of the restriction
    of $x\mapsto x^{-1}$ to~$k_n^{\times}$ or the graph of the cube map
    $x\mapsto x^3$ (the latter if the characteristic of~$k$ is
    $\geq 7$).

  \item As in~\cite{short}, the fact that our jacobian graphs do not
    contain a copy of $K_{2,3}$ is quite striking. Indeed, the
    Kővári--Sós--Turán theorem~\cite{kst} implies that \emph{every}
    $n$-vertex graph with average degree at least
    $(\sqrt 2+o(1))\sqrt n$ contains a copy of $K_{2,3}$, and work of
    Füredi~\cite{furedi} shows that this is asymptotically best possible
    (including the factor~$\sqrt{2}$). Thus, up to a small
    multiplicative factor, our jacobian graphs are as dense as possible
    among all $K_{2,3}$-free graphs, and therefore highly atypical among
    all graphs of the same edge density. In particular, an Erdős--Rényi random graph at the same density
    contains a copy of $K_{2,3}$ asymptotically almost surely (in fact,
    it has on the order of~$n^2$ such copies). The same holds in other
    natural random graph models at the same density, such as random
    regular graphs with degree $d = \sqrt n$.

  \item Condition~(iii) in part~(4) of the theorem is somewhat technical
    and will be discussed in more detail when we give the proof. It is
    related to the finer aspects of the equidistribution theorem
    of~\cite{ffk}. Although we expect that it should be true in all
    cases, we do not currently know how to check it concretely in any
    given case. On the other hand, we will see that the second condition is not very
    restrictive at all, and can be checked easily for a concrete curve.

  \end{enumerate}
\end{remark}

\begin{remark}
  (1) The Wasserstein (or Monge--Kontorovich or Rubinstein--Kontorovich)
  distance $W_1(\mu_1,\mu_2)$ between probability measures on a compact
  metric space~$X$ (in our case, the interval $[-2,2]$) can be defined as
  \[
    W_1(\mu_1,\mu_2)=\sup_{f\text{ $1$-Lipschitz}}\Bigl|
    \int_{X}fd\mu_1-\int_Xfd\mu_2\Bigr|,
  \]
  where~$f$ runs over $1$-Lipschitz functions $X\to \Cc$ (see~\cite{k-u}
  for an introduction to Wasserstein metrics from the point of view of
  equidistribution).

  (2) We expect that the bound~(\ref{it-w1}) could be improved to
  \[
    W_1\Bigl( \frac{1}{|J_{\mfm}(k_n)|-1} \sum_{\substack{\lambda\in
        \mathrm{Sp}(\Gamma_n(C,\mfm))\\\lambda\not=1}} \delta_{\lambda
      d_n/\sqrt{d_n-1}}, \mu_{\mathrm{sc}} \Bigr) \ll
    \frac{1}{|k_n|^{\gamma}},
  \]
  for some $\gamma>0$.
\end{remark}

\begin{remark}\label{rm-moduli}
  As we already stated in Remark~\ref{rm-intro}, one feature of jacobian
  graphs is that they depend on ``continuous'' parameters. More
  precisely, recall that for any genus~$g$ and field~$k$, there exists an
  algebraic variety $\mathcal{M}_g$ defined over~$k$ whose points in any
  extension~$k'$ of~$k$ correspond naturally to isomorphism classes of
  curves of genus~$g$ over~$k'$; this is called the {coarse
    moduli space} of curves of genus~$g$. The dimension
  of~$\mathcal{M}_g$ is $3g-3$ for~\hbox{$g\geq 2$}. In particular, for
  jacobian graphs associated to a curve of genus~$2$, we have
  a~$3$\nobreakdash-parameter family of examples over~$k$.

  More concretely, some interesting one-parameter families can be
  obtained, e.g., by looking at the (smooth projective model of the)
  family of curves
  \[
    y^2=f(x)(x-t)
  \]
  for a fixed squarefree polynomial $f\in k[X]$ of degree~$4$, with~$t$
  varying among the elements of~$k$ which are not roots of~$f$.

  One potential application of these continuous families has to do with averaging over many different jacobian graphs. For example, returning to \cref{qu:ind}, we expect it to be very difficult to bound the independence number of an arbitrary jacobian graph, but it is not inconceivable that one could instead bound the independence number of an average jacobian graph, using the fact that we have a rich family of these parameterized by $\mathcal M_2$. In closely related contexts, such averaging arguments were used by Ruzsa \cite{ruzsa} and Pach--Zakharov \cite{pz} to study variants of Sidon sets; in both cases, the key input is that abelian groups such as~$k \times k$ contain a rich ``continuous'' family of Sidon sets, and hence one can prove the existence of one with desirable properties by averaging over the whole family. In another closely related context, there is a major line of work that proves that many groups have Cayley graphs and Cayley sum graphs with desirable properties (in particular, small independence number, as in \cref{qu:ind}), by averaging over all possible generating sets; see the papers by Conlon\nobreakdash--Fox\nobreakdash--Pham--Yepremyan \cite{cfpy} and Alon--Pham \cite{ap}, as well as the references therein, for more details. 
\end{remark}

\subsection*{Notation.} Given complex-valued functions $f$ and $g$
defined on a set $S$, we write synonymously $f=O(g)$ or $f \ll g$ if
there exists a real number $C \geq 0$ (called an ``implicit
constant'') such that the inequality $|f(s)|\leq C |g(s)|$ holds for all
$s \in S$. We write $f\asymp g$ if $f\ll g$ and~$g\ll f$ hold. If $f$ and~$g$ are defined on some infinite subset of positive integers,
we write~$f \sim g$ if $\lim_{n \to \infty} f(n)/g(n)=1$.

\subsection*{Acknowledgements.}  We thank K. Soundararajan and
J. Merikoski for discussions related to analytic properties of
elliptic curves and D. Loeffler for remarks on elliptic curves modulo
primes.

A.\,F.\,acknowledges the support of the CDP C2EMPI, together with the
French State under the France-2030 programme, the University of Lille,
the Initiative of Excellence of the University of Lille, the European
Metropolis of Lille for their funding and support of the
R\nobreakdash-CDP-24-004-C2EMPI project. J.\,F.\,is partially supported by the
European Research Council (ERC) under the European Union's Horizon
2020 research and innovation programme (grant agreement
no.\,101170066). E.\,K.\,is partially supported by the SNF grant
$219220$ and the SNF\nobreakdash-ANR ``Etiene'' grant
$10003145$. Y.\,W.\,is supported by Dr.\,Max R\"ossler, the Walter
Haefner Foundation and the ETH Z\"urich Foundation.

\section{Sum graphs: basic properties and spectrum}

The graphs we consider in this paper are examples of the general construction of sum graphs.  We begin by surveying some of their basic properties in
general. Let~$G$ be a finite abelian group and let~$S$ be a subset
of~$G$.  The \emph{(Cayley) sum graph} $\Gamma(G,S)$ is the graph with
vertex set~$G$ and with $x$ and~$y$ connected by an edge if and only if
$x+y\in S$. This graph is $|S|$-regular (potentially with loops on at
most $\abs S$ vertices), and in particular all operators that can be
used to define the spectrum are directly equivalent up to scaling. We
will use the Markov averaging operator, which we will denote
by~$M_{G,S}$ (omitting the subscript when no confusion arises). This is
defined on functions~$f\colon G\to \Cc$ by
\begin{align*}
  M_{G,S}f(y)&=\frac{1}{|S|}\sum_{\substack{x\in G\\x+y\in S}}f(x)=
  \frac{1}{|S|}\sum_{s\in S}f(s-y)
\end{align*}
for all~$y\in G$. Throughout, the spectrum of a graph $\Gamma(G,S)$ always
refers to the spectrum of its Markov operator $M_{G,S}$, which is a
self-adjoint linear map on the space of complex-valued functions on~$G$. It is classical that one can compute this spectrum 
using characters of~$G$ (this already appears, somewhat implicitly, in the work of Chung~\cite{chung}). We state this formally now,
with a minor additional property.  We denote by~$\widehat{G}$ the group
of characters of~$G$, i.e., the group (under pointwise multiplication)
of group homomorphisms from~$G$ to~$\Cc^{\times}$.

\begin{proposition}\label{pr-1}
  Let $G$ be a finite abelian group and $S\subset G$ a non-empty subset.
  \begin{enumth}
  \item The spectrum of the graph~$\Gamma(G,S)$ consists of the
    numbers
    \[
      \frac{1}{|S|}\sum_{y\in S}\chi(y)
    \]
    for all characters~$\chi\in \widehat{G}$ such that $\chi^2=1$ and of
    the numbers
    \[
      -\Bigl|\frac{1}{|S|}\sum_{y\in S}\chi(y)\Bigr|,\quad\quad
      \Bigl|\frac{1}{|S|}\sum_{y\in S}\chi(y)\Bigr|
    \]
    for all pairs $(\chi,\bar{\chi})$ of non-real characters.
    
  \item For any~$a\in G$, the graphs $\Gamma(G,S)$ and $\Gamma(G,a+S)$
    have the same eigenvalues, with multiplicity, associated to pairs
    of non-real characters.
\end{enumth}

The statement is to be understood as follows: if some of the sums coincide for different choices of~$\chi$, then the corresponding eigenvalue occurs with
multiplicity equal to the number of characters that give rise to it.
\end{proposition}

\begin{remark}
  Many readers will also be aware of Cayley graphs, which are defined
  similarly to sum graphs, but where we insist that the set $S$ be
  symmetric (i.e., satisfies $S=-S$), and then join $x$ and $y$ by an
  edge if and only if $x-y \in S$. Most of our results immediately carry
  through for Cayley graphs, since their spectra are given by the same
  character sums above (without taking absolute values). However, we
  work with Cayley sum graphs because they are defined for all
  $S \subset G$, rather than only symmetric subsets. Another difference,
  which may be relevant to certain applications, is that whereas Cayley
  graphs are highly symmetric (the group $G$ acts transitively on the
  vertices), this is usually not the case for sum graphs.
\end{remark}

The short proof of Proposition~\ref{pr-1} is given
in~\cite{short}*{Prop.\,3.3}. We repeat it in a more general context,
considering general linear operators on compact abelian groups with
additive kernel, since it is no more difficult, and it might be useful for
other purposes.

Let~$G$ be a compact abelian group. We denote by $dx$ the unique
probability Haar measure on~$G$ (so it is the normalized counting
measure if~$G$ is a finite group). Let $t\colon G\to\Cc$ be a bounded
measurable function on~$G$, and let~$T$ be the integral operator with
kernel $(x,y)\mapsto t(x+y)$, acting on the Hilbert space~$L^2(G)$. That is,
\[
  Tf(y)=\int_Gt(x+y)f(x)dx
\] 
for $f\in L^2(G)$ and $y\in G$. We will generalize Proposition~\ref{pr-1} by finding a ``spectral
decomposition'' of~$T$ (note that~$T$ is not always diagonalizable).

For the compact group~$G$, the dual group~$\widehat{G}$ is the group of
continuous group morphisms~$G\to\Cc^{\times}$; since~$G$ is compact,
every such character takes values in the unit circle.

For every character~$\chi\in\widehat{G}$ and any~$y\in G$, we get
\[
  T\chi(y)=\int_G t(x+y)\chi(x)dx=\overline{\chi(y)}\,\widehat{t}(\chi), 
\]
where $\widehat{t}$ is the Fourier transform of~$t$ on~$G$, defined as
\[
  \widehat{t}(\chi)=\int_G t(x)\chi(x)dx
\]
for $\chi\in\widehat{G}$ (for convenience, we use here the
``probabilistic'' normalization of the Fourier transform,  instead of the more usual analytic one involving
$\overline{\chi(x)}$). Hence, we get the equation
\[
  T\chi=\widehat{t}(\chi)\overline{\chi}.
\]

There are now two cases.
\par
\medskip
\par
\textbf{Case 1.} If~$\chi^2=1$, then $\chi=\overline{\chi}$ is an eigenvector with
eigenvalue $\widehat{t}(\chi)$.

\par
\medskip
\par
\textbf{Case 2.} If~$\chi^2\not=1$, then $\chi$
and~$\overline{\chi}$ span a $2$-dimensional subspace~$E_{\chi}$ of~$L^2(G)$ satisfying \hbox{$T(E_{\chi})\subset E_{\chi}$} (by linear
  independence of characters, or simply because $\overline{\chi}=\alpha\chi$
  implies~$\alpha=1$ by evaluating at~$1$). In
the basis $(\chi,\overline{\chi})$, the matrix of the linear map induced
by~$T$ on~$E_{\chi}$ is
\[
  \begin{pmatrix}
    0&\widehat{t}(\overline{\chi})\\
    \widehat{t}(\chi)&0
  \end{pmatrix}.
\]

We now appeal to a simple lemma from linear algebra.

\begin{lemma}
  Let~$\alpha$ and~$\beta$ be complex numbers and
  $A_{\alpha,\beta}= \begin{pmatrix}
    0&\beta\\
    \alpha&0
  \end{pmatrix}$.
  
  \begin{enumth}
  \item The spectrum of~$A_{\alpha,\beta}$ is
    $\{\sqrt{\alpha\beta},-\sqrt{\alpha\beta}\}$, where the choice of
    the square root is irrelevant if the product $\alpha \beta$ is not a positive real
    number.
    
  \item The matrix $A_{\alpha,\beta}$ is diagonalizable unless exactly
    one of $\alpha$ and~$\beta$ is zero.

  \item Let us endow~$\Cc^2$ with the standard inner product. Then the matrix
    $A_{\alpha,\beta}$ is normal if and only if $|\alpha|=|\beta|$, it
    is self-adjoint if and only if $\beta=\overline{\alpha}$, and it is
    unitary if and only if $|\alpha|=|\beta|=1$.
  \end{enumth}
\end{lemma}

\begin{proof}
  The characteristic polynomial $\det(X-A_{\alpha,\beta})$ is
  $X^2-\alpha\beta$, and the adjoint (for the standard inner product) is
  $A_{\bar{\beta},\bar{\alpha}}$, so all these facts follow from
  elementary linear algebra.
\end{proof}

This lemma implies that the spectrum of the linear operator~$T$ consists
of the numbers~$\widehat{t}(\chi)$ for $\chi$ of order~$2$, and the
numbers
\[
  \sqrt{\widehat{t}(\chi)\widehat{t}(\overline{\chi})},\quad\quad
  -\sqrt{\widehat{t}(\chi)\widehat{t}(\overline{\chi})}
\]
for each pair $(\chi,\overline{\chi})$ with $\chi^2\not=1$.  Moreover,
$T$ is normal if and only if the two-by-two blocks are all normal, meaning
if and only if $ |\widehat{t}(\chi)|=|\widehat{t}(\overline{\chi})|$ for
all~$\chi$ with $\chi^2\not=1$ (or for all~$\chi$, since this condition
is automatic for $\chi^2=1$).
Similarly, the operator $T$ is self-adjoint if and only if
$\overline{\widehat{t}(\chi)}=\widehat{t}(\overline{\chi})$ for
all~$\chi$. Since
\[
  \overline{\widehat{t}(\chi)}=\int_G
  \overline{t(x)}\overline{\chi(x)}dx,\quad\quad
  \widehat{t}(\overline{\chi})=
  \int_Gt(x)\overline{\chi(x)}dx,
\]
this is the case if and only if~$t$ is real-valued. In this case, the
operator~$T$ is diagonalizable; its
spectrum consists of $\widehat{t}(\chi)$ for $\chi^2=1$ and
\[
  |\widehat{t}(\chi)|,\quad\quad - |\widehat{t}(\chi)|
\]
for pairs $(\chi,\overline{\chi})$ with $\chi^2\not=1$, with
multiplicity taken into account.

We now specialize this discussion to prove Proposition~\ref{pr-1}. We
consider therefore a finite group~$G$, a non-empty subset~$S\subset G$, and $t$ of the form
\[
  t=\frac{\mathbf 1_S}{|S|},
\]
where~$\mathbf 1_S$ denotes the characteristic function of~$S$. Then the
operator~$T$ is the Markov operator~$M_{G,S}$ of the graph
$\Gamma(G,S)$, and the previous discussion establishes the first two
statements of Proposition~\ref{pr-1}.

For the last statement, observe that for any character~$\chi$ of~$G$, we
have
\[
  \frac{1}{|S|} \sum_{y\in a+S}\chi(y)=
  \frac{\overline{\chi(a)}}{|S|}\sum_{y\in S}\chi(y),
\]
and thus the two sums have the same modulus for any non-real character.

\section{Jacobian graphs: basic properties}\label{sec-graphs}

Let $k$ be a finite field. We now consider the data $(k,C,\mfm)$ which is used to define the jacobian graphs $\Gamma_n(C,\mfm)$.  For
simplicity, we will assume that all components of the modulus~$\mfm$ are
$k$-rational points. For each $n \geq 1$, we abbreviate
\[
S_n=i_{\delta}((C\setminus\mfm)(k_n)) \subset J_{\mathfrak{m}}(k_n). 
\]

From the Riemann hypothesis for curves over finite fields, due to Weil,
and the structure of the generalized jacobians, it follows that
$|J_{\mathfrak{m}}(k_n)|\sim |k_n|^{\dim(J_{\mfm})}$ and~$|S_n|\sim |k_n|$ as
$n\to+\infty$ (see, for instance,~\cite{ffk2}*{\S\,3} for precise bounds
and references for these facts; for curves and for the
case~$\mfm=\emptyset$, see also~\cite{milne1}*{Th.\,19.1} and
\cite{milne2}*{Th.\,11.1}). Thus, the degree~$d$ and the number of
vertices~$|V|$ of~$\Gamma_n(C,\mfm)$ are related by
$d\sim |V|^{1/\dim(J_{\mfm})}$.

From Proposition~\ref{pr-1} applied to~$\Gamma_n(C,\mfm)$, it follows
that the spectrum of~$\Gamma_n(C,\mfm)$ is determined by the
exponential sums
\[
  \sum_{x\in S_n}\chi(x)=\sum_{x\in
    (C\setminus\mfm)(k_n)}\chi(i_{\delta}(x))
\]
for~$\chi\in\widehat{G}_n$. We deduce the following spectral bound.

\begin{proposition}\label{pr-max-spectrum}
  Let $(C,\mfm)$ be as above. There exists a constant~$c\geq 0$ with the
  following property: if $n\geq 1$ satisfies
  \[
    |(C\setminus\mfm)(k_n)|>c|k_n|^{1/2},
  \]
  then the graph $\Gamma_n(C,\mfm)$ is connected, so $\lambda=1$ is an
  eigenvalue with multiplicity~$1$, and any eigenvalue $\lambda\neq 1$
  of the Markov operator of~$\Gamma_n(C,\mfm)$ satisfies
  \[
    |\lambda|\leq c\frac{|k_n|^{1/2}}{|(C\setminus\mfm)(k_n)|}
    =O(1/|k_n|^{1/2}).
  \]
\end{proposition}

\begin{proof}
  We first recall that a graph is connected if and only if~$1$ is an
  eigenvalue with multiplicity~$1$.

  The statements follow from the Riemann hypothesis over finite fields,
  but in the most general form proved by Deligne~\cite{weil2}. More
  precisely, the latter implies (see~\cite{ffk}*{Ch.\,11}, or compare
  with the beginning of the proof of Theorem~\ref{th-equi}
  below) that there exists a constant~$c\geq 0$ such that
  \[
    \Bigl|\frac{1}{|k_n|^{1/2}}\sum_{x\in
      (C\setminus\mfm)(k_n)}\chi(i_{\delta}(x))
    \Bigr|\leq c
  \]
  for any non-trivial
  character~$\chi\colon J_{\mfm}(k_n)\to\Cc^{\times}$. 

  If a character $\chi\not=1$ gives rise to the eigenvalue~$1$ of the
  Markov operator, then the corresponding sum must be equal to
  \[
    \Bigl|\frac{1}{|k_n|^{1/2}}\sum_{x\in
      (C\setminus\mfm)(k_n)}\chi(i_{\delta}(x)) \Bigr|
    =\frac{|(C\setminus\mfm)(k_n)|}{|k_n|^{1/2}}, 
  \]
and hence the following inequality holds: 
  \[
    \frac{|(C\setminus\mfm)(k_n)|}{|k_n|^{1/2}}\leq c.
  \]

  Therefore, for all~$n$ for which the above inequality fails (in particular for
  all large~$n$, since~\hbox{$|(C\setminus \mfm)(k_n)|\sim |k_n|$} as
  $n\to +\infty$), we deduce that $1$ is an eigenvalue with
  multiplicity~$1$, and that the other eigenvalues
  satisfy
   \[
    |\lambda|=\Bigl|\frac{1}{|S_n|}\sum_{x\in S_n}\chi(x)\Bigr|
    =\frac{|k_n|^{1/2}}{|S_n|} \Bigl|\frac{1}{|k_n|^{1/2}}\sum_{x\in
      (C\setminus \mfm)(k_n)}\chi(i_{\delta}(x))\Bigr|\leq c
    \frac{|k_n|^{1/2}}{|S_n|},
  \]
  for some character~$\chi$.
\end{proof}

In the remainder of this paper, we will restrict our attention to pairs
$(C,\mfm)$ such that the graphs have maximal possible edge density and
$C$ has genus~$g\geq 1$. As explained in~\cite{ffk2}*{\S\,3}, this occurs when
$g=1$ and $\mfm$ has degree~$2$, or $g=2$ and $\mfm$ is trivial, and the
image of~$C\setminus\mfm$ in~$J_{\mfm}$ is then a symmetric Sidon set.

We will now describe the three possible cases of data $(C,\mfm)$ that 
satisfy these conditions. We first recall that in the case~$g=1$,
since $k$ is a finite field, there exists a $k$-rational point which can
be taken as the origin to define an abelian group law on~$C$, which
becomes an elliptic curve. We always assume that this is done. 

\begin{enumerate}
\item If $g=1$ and $\mfm$ is the sum of two distinct $k$-rational
  points, then the group $J_{\mfm}(k_n)$ sits for all $n\geq 1$ in an
  exact sequence
  \begin{equation}\label{eq-exact-1}
    1\to k_n^{\times}\to J_{\mfm}(k_n)\to C(k_n)\to 1. 
  \end{equation}
  
\item If $g=1$ and $\mfm$ is a single $k$-rational point with
  multiplicity two, then the group~$J_{\mfm}(k_n)$ sits for all
  $n\geq 1$ in an exact sequence
  \begin{equation}\label{eq-exact-2}
    1\to k_n\to J_{\mfm}(k_n)\to C(k_n)\to 1.
  \end{equation}
  
\item If $g=2$ and $\mfm$ is zero, then $J_{\mfm}$ is the classical
  two-dimensional jacobian variety of the smooth curve $C$. We then write $J$ instead of $J_{\mfm}$. 
\end{enumerate}

We now use this information to address the question of which
integers arise as the number of vertices for jacobian graphs satisfying
the assumptions of Theorem~\ref{th-jacobian}. Readers mostly interested
in the proof of the theorem for an individual curve may skip to the next
section.

\begin{proposition}\label{pr-n}
  Let~$x\geq 2$ be a real number. Let $N(x)$ \respup{$M(x)$} be the
  number of positive integers~$m\leq x$ for which there exist a
  prime~$p$ and an ordinary curve of genus~$2$ over~$\Ff_p$ such that $J(\Ff_p)$ has $m$ points  \respup{the number
    $m\leq x$ for which such a curve exists for which the center of
    the symmetric Sidon set is in $2J(\Ff_p)$}.

  We then have
  \[
    N(x)=x+O(x^{5/6}),\quad\quad M(x)\geq \frac{x}{2}+O(x^{5/6}).
  \]
\end{proposition}

\begin{proof}
  The second statement follows from the first by simply
  considering the odd integers counted by $N(x)$. Indeed, if $J(\Ff_p)$ has odd cardinality, then the multiplication
  by~$2$ map is an isomorphism of $J(\Ff_p)$, so the center of the
  symmetric Sidon set is in its image.
  
  Let $N'(x)$ be the number  of positive integers~$m\leq x$ for which there exists an \emph{ordinary abelian surface} over a finite field~$\Ff_p$ with $p>7$ having $m$ rational points. It was proved by Bröker, Howe, Lauter and
  Stevenhagen~\cite{bhls}*{Th.\,3.1} that $N'(x)$ satisfies\footnote{\,A key analytic input is a theorem of
    Matom\"aki~\cite{mato} concerning the average size of ``large'' gaps
    between primes; using a more recent result of
    Heath-Brown~\cite{hb}, we would obtain an error term of
    size~$x^{4/5}$.}
  \[
    N'(x)=x+O(x^{5/6})
  \]
  (to be precise, the fact that the surfaces are ordinary is not
  formally stated, but is mentioned in the first paragraph of the
  proof of~\cite{bhls}*{Lemma\,3.2}).

  A criterion of Howe, Nart and Ritzenthaler~\cite{hnr}*{Th.\,1.2}
  determines which isogeny classes of abelian surfaces over finite
  fields contain the jacobian of a genus~$2$ curve. Since isogenous
  abelian varieties over finite fields have the same number of
  rational points, it suffices therefore to prove that the abelian
  surfaces constructed in~\cite{bhls}*{Lemma\,3.2} satisfy these
  conditions except possibly for $\ll x^{5/6}$ values of~$m$.

  For a given~$m$, let~$A$ be the ordinary abelian surface constructed
  in~\cite{bhls}*{Proof\,of\,Lemma\,3.2}.  Recall, for instance from \cite{milne1}*{Th.\,19.1}, that the size of
  $A(\Ff_p)$ is the value $f(1)$ of the so-called Weil polynomial
  of~$A$. The Weil polynomials
  for the abelian surface constructed in~\cite{bhls} are expressed in
  the form
  \[
    f = X^4 -aX^3+ (b + 2p)X^2 - apX + p^2
  \]
  for some integers~$a$ and~$b$ (see~\cite{bhls}*{(3.1)}).  On the other
  hand, the criterion of~\cite{hnr} concerns a Weil polynomial of the
  form
  \[
    f=X^4+a_1X^3+b_1X^2+a_1pX+p^2
  \]
  for integers~$a_1$ and~$b_1$.

  Suppose that the isogeny class of~$A$ does \emph{not} contain a
  jacobian.  Then~\cite{hnr}*{Th.\,1.2} proves that one of the
  following two situations occurs.
  
  \textbf{Case 1.} \emph{The abelian surface splits (up to isogeny) as a
    product of two elliptic curves.} Then, by~\cite{hnr}*{Th.\,1.2}, the
  associated Weil polynomials are of the form
  \[
    (X^2-sX+p)(X^2-tX+p)
  \]
  for some integers~$s$ and~$t$, and we have either $|s-t|=1$ or $s=t$
  (see~\cite{hnr}*{table\,1}, noting that the ordinary case corresponds
  to $p$-rank equal to~$2$).

  If $s=t$, then $m=(p-s+1)^2$ is a square, so this can only occur for
  $\ll x^{1/2}$ values of~$m$. If $s-t=1$ (and symmetrically if
  $s-t=-1$), then
  \[
    m=(p-(t+1)+1)(p-t+1)=(p-t+1)^2-(p-t+1)
  \]
  is of the form $u^2-u$ for some integer~$u$, and this can only occur
  for $\ll x^{1/2}$ values of~$m$.
  \par
  
  \textbf{Case 2.} \emph{The abelian surface is simple.}  Then
  by~\cite{hnr}*{table\,2}, we first see that $b_1<0$ (this is part of
  the information from the first three lines of this table).  However,
  the construction in~\cite{bhls}*{Lemma\,3.2} for primes $p>7$
  provides polynomials with $b\geq -2p-2$, and hence $b_1=b+2p\geq -2$. Thus, the only possibilities would be $b_1=-1$ or
  \hbox{$b_1=-2$}. The second case is not possible for~$p>7$, because
  from~\cite{hnr}*{table\,2}, we have either \hbox{$b_1=1-2p\leq -13$}, or
 \hbox{$b_1=2-2p\leq -12$}, or all primes dividing~$b_1$ are
  $\equiv 1\mods{3}$. In the first case, from \cite{hnr}*{table\,2},
  we see that $a_1^2=p+b_1=p-1$. Thus, the prime $p$ must be of the
  form $s^2+1$ for some integer~$s\geq 1$, and for such a prime, the
  Weil polynomial of~$A$ can only be one of the two polynomials
  \[
    X^4+sX^3-X^2+psX+p^2,\quad\quad X^4-sX^3-X^2-psX+p^2.
  \]

  Note that~$p$ is of size $O(m^{1/2})$ by the estimate $|A(\Ff_p)|\sim |\Ff_p|^{\dim(A)}$; the number of primes~$p$ of the
  form $s^2+1$ that are of this size is $O(m^{1/4})$ (just counting all
  integers which can be of the form $s^2+1$), and each prime leads to at
  most two excluded values of $m=f(1)$. Hence, the number of $m\leq x$
  where this second case may occur is $O(x^{1/4})$.
\end{proof}

\begin{remark}\label{rm-n}
  (1) We note that standard conjectures on gaps between primes suggest
  strongly that \emph{all} positive integers (except maybe squares or integers of the form
  $r(r\pm 1)$ and finitely many exceptions) should arise as the size of a
  jacobian graph.
  
  (2) Here are some comments on values of the orders of jacobian graphs
  in genus~$1$. For simplicity, we focus on the cases where the base
  field is~$\Ff_p$ for some prime~$p$.

  \begin{enumerate}
  \item If $g=1$ and $\mfm$ is a double point, then from the exact
    sequence~(\ref{eq-exact-2}), we see that the number of vertices is
    $p|C(\Ff_p)|$, while the degree is~$|C(\Ff_p)|-1$. It follows from
    the Hasse bound that $|C(\Ff_p)|=p+1-a_C$, with
    $|a_C|\leq 2\sqrt{p}$, and from the work of Deuring and Honda--Tate
    that every integer~$a$ satisfying this condition occurs for some
    suitable curve~$C$ (see, for instance, the
    accounts by Waterhouse in~\cite{waterhouse}*{Th.\,4.1} and by Serre
    in~\cite{serre-points}*{\S\,2.6}). So the integers we obtain for
    these jacobian graphs are very special, since they are those of the
    form $pn$ for some prime~$p$ and some integer $n$ satisfying
    $|n-1-p|\leq 2\sqrt{p}$, whereas it is well-known that most integers
    do not have a prime factor of size close to their square root.

  \item If $g=1$ and $\mfm$ is a sum of two distinct points, from the
    exact sequence~(\ref{eq-exact-1}) we see that the number of vertices
    is $(p-1)|C(\Ff_p)|$, while the degree is~$|C(\Ff_p)|-2$. So we
    again have very special integers, those of the form $nm$ with $n+1$
    prime and~$|n+2-m|\leq 2\sqrt{n+1}$.
  \end{enumerate}
\end{remark}

Finally, we conclude this section by commenting on the issue of the
``effectivity'' of our construction.  There are a number of aspects:
\begin{enumerate}
\item Given a prime number~$p$, it is straightforward to ``write down''
  a curve~$C$ over $\Ff_p$ of genus~$1$ or~$2$, using equations of the
  form $y^2=f(x)$, and to find rational points to describe a
  modulus~$\mfm$ (in the genus~$1$ case). Because curves of genus~$2$
  and their jacobians are extensively used in cryptography, there are
  many well-developed algorithms to work with them, and thus to
  construct the corresponding jacobian graphs efficiently. The case of
  generalized jacobians is less developed, but Déchène's
  thesis~\cite{dechene}*{\S\,5.3} provides all basic algorithms for this
  purpose. Given the size of~$p$, the number of vertices of the jacobian graph associated to~$C$ is well-understood by the Riemann hypothesis over
  finite fields, and is of order~$p^2+O(p^{3/2})$.
  
\item One might also be interested in constructing a
  jacobian graph with a \emph{given} number~$n$ of vertices. If~$n$ can
  be factored and has the form given in Remark~\ref{rm-n}\,(2), this can
  be done efficiently using elliptic curves (see the references in~\cite{bhls}*{\S\,2}). On the other hand, the
  work of Bröker, Howe, Lauter and
  Stevenhagen~\cite{bhls}*{Th.\,1.1,\,\S\,3} indicates that the
  corresponding problem for genus~$2$ is more difficult.
\end{enumerate}

\begin{remark}
  The case~$g=0$ and $\deg(\mfm)=3$ also gives graphs with
  $\dim(J_{\mfm})=2$, which do have some interest, but these are in fact
  classical. Indeed, as shown in~\cite{ffk2}, the
  sets~$S_n$ then correspond to the
  five known infinite families of ``densest'' Sidon sets (see, for
  instance, the discussion of Eberhard and Manners~\cite{e-m}), and
  moreover, their spectrum is much simpler, consisting of only a bounded
  number of eigenvalues as $n$ varies.  More precisely, one can show
  the following (see previous discussions, e.g., in~\cites{mw,ks}): the
  spectrum of~$\Gamma_n(C,\mfm)$ in this case consists of
  \begin{enumerate}
  \item the eigenvalue~$1$ with multiplicity~$1$; 
  \item two eigenvalues of size asymptotic to $-1/|k_n|^{1/2}$
    and~$1/|k_n|^{1/2}$ with multiplicity~$\sim |k_n|^2/2$;
  \item a single eigenvalue of size~$O(1/|k_n|)$ with
    multiplicity~$\sim |k_n|$.
  \end{enumerate}

  In terms of exponential sums, using the descriptions of the sets~$S_n$
  in~\cite{ffk2}, this boils down to the fact that the sums
  \[
    \sum_{\substack{x\in k_n^{\times}\\x\not=1}}\chi_1(x)\chi_2(1-x),
    \quad\quad
    \sum_{x\in k^{\times}}\chi(x)\psi(ax),\quad\quad
    \sum_{x\in k_n}\psi(ax+bx^2),
  \]
  (where $\chi$, $\chi_1$, $\chi_2$ are multiplicative characters
  of~$k_n$,  $\psi$ is a non-trivial additive character of~$k_n$, and
  $(a,b)\in k_n^2$) are all ``generically'' of modulus exactly
  $\sqrt{|k_n|}$ (these are, respectively, Jacobi sums, Gauss sums and
  quadratic additive Gauss sums).
    
  In this case, there is no further interesting information to be
  gathered from the distribution of the spectrum, which is very rigid:
  multiplying the non-trivial eigenvalues by $\sqrt{|k_n|}$, the limit
  of the empirical spectral measures is simply the atomic measure
  $\demi (\delta_{-1}+\delta_1)$.
\end{remark}

\section{Jacobian graphs: semicircular spectrum}

We use the basic notation from the previous section. 
We recall here the main equidistribution result of~\cite{ffk}*{Ch.\,4},
as applied to the situation corresponding to the jacobian graphs we
consider (this is partly sketched at the end
of~\cite{ffk}*{\S\,11.1}). The outcome of the general theory, combined
with some specific analysis of the current situation, is the following
equidistribution property. For $n\geq 1$ and $\chi\in
\widehat{J}_{\mfm}(k_n)$, we denote
\[
  U_n(\chi)=\frac{1}{|k_n|^{1/2}}\sum_{x\in (C\setminus \mfm)(k_n)}
  \chi(i_{\delta}(x))
\]

Except for the normalization, these are the exponential sums that
determine the spectrum of~$\Gamma_n(C,\mfm)$ by Proposition~\ref{pr-1}.

\begin{theorem}\label{th-equi}
  Let~$(C,\mfm)$ be such that $\dim(J_{\mfm})=2$ and $g=1$ or~$2$.
  Assume that the center of the symmetric Sidon set
  $i_{\delta}(C\setminus\mfm)$ is~$0\in J_{\mfm}$. If~$g=1$ and $\mfm$
  is a double point, assume further that the characteristic of~$k$ is
  $\geq 5$.

  There exist two compact Lie groups $K^g$ and~$K^a$, contained in
  $\SU_2(\Cc)$, with $K^g$ a normal subgroup of~$K^a$, such that the
  following properties hold:
  \begin{enumth}
  \item\label{it-equid1} For any continuous and bounded function
    $f\colon \Cc\to \Cc$, we have
    \[
      \lim_{N\to +\infty}\frac{1}{N}\sum_{1\leq n\leq N}
      \frac{1}{|J_{\mfm}(k_n)|}\sum_{\chi\in \widehat{J}_{\mfm}(k_n)}
      f(U_n(\chi))=\int_{K^a}f(\Tr(x))dx.
    \]
    
  \item If $K^a=K^g$, then for any continuous and bounded function
    $f\colon \Cc\to \Cc$, we have
    \[
      \lim_{n\to +\infty} \frac{1}{|J_{\mfm}(k_n)|}\sum_{\chi\in
        \widehat{J}_{\mfm}(k_n)} f(U_n(\chi))=\int_{K^a}f(\Tr(x))dx.
    \]

  \item\label{it-equid3} Furthermore, either $K^a=K^g=\SU_2(\Cc)$ or $K^a$ is a finite
    subgroup of~$\SU_2(\Cc)$ isomorphic to one of $\SL_2(\Ff_3)$,
    $\GL_2(\Ff_3)$ or $\SL_2(\Ff_5)$ embedded in $\SU_2(\Cc)$ by an
    irreducible representation.
  \end{enumth}

  In these statements, $dx$ denotes the unique probability Haar measure
  on the group~$K^a$.
\end{theorem}

\begin{proof}
  The \emph{a priori} equidistribution (i.e., the existence of compact
  Lie groups $K^a$ and $K^g$, with $K^g$ a normal subgroup of $K^a$ and
  $K^a\subset \Un_r(\Cc)$ for some integer~$r\geq 0$ for which (1)
  holds) is one of the main results of~\cite{ffk}, specifically of
  Theorem~4.8 of \loccit, applied to the object
  $M=i_{\delta*}\bQl[1](1/2)$ of (the derived category of constructible
  sheaves on) $J_{\mfm}$ whose trace function is the (normalized)
  characteristic function of the image of~$i_{\delta}$. If $K^a=K^g$, we
  get the similar \emph{a priori} statement (2) from Theorem~4.15 of \loccit 

  To conclude the proof of the theorem, we first check that the
  integer~$r$ here is equal to~$2$. Recall that, for any
  $n\geq 1$ and character~$\chi\in \widehat{J}_{\mfm}(k_n)$, there is an
  associated lisse character sheaf $\mcL_{\chi}$ on the generalized
  jacobian $J_{\mfm}$ (over $k_n$) with trace function given
  by~$\chi$. The theory (see~\cite{ffk}*{Prop.\,3.17}) implies that
  there exist subsets $\mcX_n$ of characters of
  $\widehat{J}_{\mfm}(k_n)$ such that the equality
  \[
    r=\dim H^0_c(J_{\mfm,\bar{k}},M\otimes\mcL_{\chi})=
    \dim H^1_c((C\setminus \mfm)_{\bar{k}},\mcL_{\chi}|(C\setminus\mfm)),
  \]
  holds for all $\chi\in\mcX_n$, and $|\mcX_n|\sim |J_{\mfm}(k_n)|$ as $n\to
  +\infty$. 
 
  Since~$C$ is a curve, this dimension can be computed by means of the Grothendieck--Ogg--Shafarevich formula for the Euler--Poincaré characteristic (see, e.g.,~\cite{ffk}*{Th.\,C.2}): for any
  character~$\chi\not=1$ in $\mcX_n$ one has
  \[
    \dim H^1_c((C\setminus
    \mfm)_{\bar{k}},\mcL_{\chi}|(C\setminus\mfm))=(2g-2+|\mfm|)+\sum_{x\in \mathrm{supp}(\mfm)}\swan_x(\mcL_{\chi}|(C\setminus\mfm) )
  \]
  (this is because both $H^0_c$ and $H^2_c$ vanish if $\chi$ is
  non-trivial, the latter due to the fact that the image of~$i_{\delta}$
  generates $J_{\mfm}$, and hence $\mcL_{\chi}$ cannot be geometrically trivial on the
  image of $i_{\delta}$), where the first term arises as the rank
  of~$\mcL_{\chi}$, which is~$1$, multiplied by minus the
  Euler--Poincaré characteristic of~$C\setminus\mfm$.
  
  We now consider separately the different cases of $(g,\mfm)$ which may
  occur.
  
  \begin{enumerate}
  \item If~$g=1$ and the support of $\mfm$ consists of two distinct
    points, then the Swan conductors vanish (because the character sheaf
    is tame), and hence
    \[
      r=\dim H^1_c((C\setminus
      \mfm)_{\bar{k}},\mcL_{\chi}|(C\setminus\mfm))=(2-2+2)=2
    \]
    for every non-trivial $\chi$. 
        
  \item If~$g=1$ and the support of $\mfm$ consists of one double point,
    say~$x_0$, then the formula
    \[
      \swan_{x_0}(\mcL_{\chi}|C\setminus\{x_0\} )=
      \begin{cases} 0&\text{ if $\chi$ is trivial on the kernel of
          $J_{\mfm}\to C$},
        \\
        1&\text{ if $\chi$ is non-trivial on the kernel of
          $J_{\mfm}\to C$}.
      \end{cases}
    \]
    holds by Lemma~\ref{lm-swan} below (we use here the assumption that
    the characteristic is at least $5$), and hence
    \[
      \dim H^1_c((C\setminus
      \mfm)_{\bar{k}},\mcL_{\chi}|(C\setminus\mfm))=(2-2+1)+1=2,
    \]
    for $\chi$ non-trivial on the kernel of the projection, which
    implies again that $r=2$ (since there are only $|C(k_n)|$ characters
    trivial on the kernel of the projection, so $\mcX_n$ is not
    contained in this set for large enough $n$).
    
  \item If~$g=2$ and $\mfm=\emptyset$, then there is no Swan conductor to
    consider, and hence
    \[
      r=\dim H^1_c(C_{\bar{k}},\mcL_{\chi}|C)=(4-2)=2.
    \]
    for every non-trivial $\chi$. 
  \end{enumerate}

  Thus, we have shown that $K^g\subset K^a\subset \Un_2(\Cc)$ in all
  cases. Furthermore, the assumption that the image of~$i_{\delta}$ is a
  symmetric Sidon set with center~$0$ implies that $K^a$ is contained
  in~$\SU_2(\Cc)$ (see the proof of~\cite{ffk}*{Th.\,11.1\,(2)} for this
  argument, or note that it implies that the sums $U_n(\chi)$ are real).

  Moreover, because the image of~$i_{\delta}$ is a symmetric Sidon set,
  the fourth moment
  \[
    M_4=\int_{K^a}\lvert\Tr(x)\rvert^4dx
  \]
  of $K^a\subset\SU_2(\Cc)$ is equal to~$2$ (see Proposition 8.9 and
  the end of the proof of Theorem~11.1 of \loccit for this; in
  particular, note that although one would expect a fourth moment
  equal to~$3$ by a naive computation, the contribution of the trivial
  character means that the actual fourth moment is~$2$). Larsen's alternative (see~\cite{katz-larsen}*{Th.\,1.1.6}
  or~\cite{ffk}*{Th.\,8.5}) then implies that either \hbox{$K^a=\SU_2(\Cc)$},
  or~$K^a$ is a finite subgroup of~$\SU_2(\Cc)$. In the first case,
  we deduce that $K^a=K^g=\SU_2(\Cc)$, because one knows that
  $K^g$ is a normal subgroup of~$K^a$ with abelian quotient
  (see~\cite{ffk}*{Prop.\,3.41}), and $\SU_2(\Cc)$ has no such subgroup
  except itself.

  If $K^a$ is finite, then Katz has
  proved~\cite{katz-larsen}*{Th.\,1.3.2} that the condition
  $M_4=2$ implies that the representation $K^a\subset \SU_2(\Cc)$ is
  not induced, which means that~$K^a$ is then an irreducible primitive
  subgroup of~$\SU_2(\Cc)$. The classification of these subgroups is
  classical (it is a variant of the well-known classification of
  finite subgroups of $\SO_3$). The three possible primitive groups
  are often called the binary tetrahedral, octahedral and icosahedral
  groups, but to check that these are indeed the three groups we
  indicated, it is enough to check that the orders correspond and that
  $\SL_2(\Ff_3)$, $\GL_2(\Ff_3)$ and $\SL_2(\Ff_5)$ have faithful
  two-dimensional representations, since we know that $\SU_2(\Cc)$
  contains no other finite primitive subgroup; we refer the reader, e.g,
  to~\cite{leuschke-wiegand}*{Th.\,6.11} for this classification, and
  the other facts are elementary.  Thus the last part of assertion (3)
  follows.
\end{proof}

Our next step is a conditional deduction of Theorem~\ref{th-jacobian}
from Theorem~\ref{th-equi}.

\begin{proof}[Proof of~Theorem~\ref{th-jacobian}, assuming
  $K^a=K^g=\SU_2(\Cc)$ in Theorem~\ref{th-equi}]
  \phantom{}
  \par
  (1) The property that $\Gamma_n(C,\mfm)$ does not contain a copy of
  $K_{2,3}$ follows immediately from the fact that
  $i_{\delta}((C\setminus \mfm)(k_n))$ is a symmetric Sidon set
  (by~\cite{ffk2}) and an elementary computation, which is given for
  instance in~\cite{short}*{Prop.\,3.1}.
  
  (2) Let $\chi$ be a non-trivial character of $J_{\mfm}(k_n)$. From
  the fact that the object $i_{\delta*}\bQl[1](1/2)$ is mixed of
  weights $\leq 0$, so that $i_{\delta*}\bQl[1](1/2)\otimes
  \mcL_{\chi}$ is also mixed of weights~$\leq 0$ with trace function
  on~$J_{\mfm}(k_n)$ given by
  \[
    y\mapsto \begin{cases} \chi(i_{\delta}(x)) &\text{ if }
      y=i_{\delta}(x)\text{ for some } x\in (C\setminus \mfm)(k_n),
      \\
      0&\text{ otherwise},
    \end{cases}
  \]
  and that
  \[
    H^2_c((C\setminus\mfm)_{\bar{k}},\mcL_{\chi}|(C\setminus\mfm))=0
  \]
  (because the image of~$i_{\delta}$ generates $J_{\mfm}$), we first
  deduce the bound
  \[
    \Bigl|\sum_{x\in (C\setminus\mfm)(k_n)}\chi(i_{\delta}(x))
    \Bigr|\leq 2\sqrt{|k_n|}
  \]
so that the constant~$c$ of
  Proposition~\ref{pr-max-spectrum} can be taken to be equal to~$r=2$.
  It follows that the largest absolute value of a non-trivial
  eigenvalue of $\Gamma_n(C,\mfm)$ is at most
  $2\frac{\sqrt{|k_n|}}{d_n}$.  Since the Riemann hypothesis for curves
  implies that
  \[
    d_n=|k_n|-|\mfm|+O(|k_n|^{1/2}),
  \]
  we obtain Theorem~\ref{it-ram1}.

  (3) We now prove Theorem~\ref{it-ram2}. For this, denote again
  $d_n=|(C\setminus\mfm)(k_n)|$. Since we saw that the non-trivial
  eigenvalues of the Markov operator have absolute value
  \[
    \leq \frac{2\sqrt{|k_n|}}{d_n},
  \]
  the graph~$\Gamma_n(C,\mfm)$ is a Ramanujan graph if
  \[
    \frac{2\sqrt{|k_n|}}{d_n}\leq 2\frac{\sqrt{d_n-1}}{d_n},
  \]
  i.e. if $d_n\geq |k_n|+1$. We write
  \[
    d_n=|C(k_n)|-|\mfm|=|k_n|+1-|\mfm|-a_C(k_n),
  \]
  so that it suffices to have $a_C(k_n)\leq -|\mfm|$ to obtain the
  desired inequality. According to the Riemann hypothesis for counting
  points on curves over finite fields, we can write
  \[
    a_C(k_n)=
    \begin{cases}
      2|k_n|^{1/2}\cos(2\pi n\theta_1),&\text{ if }
      g=1,\\
      2|k_n|^{1/2}(\cos(2\pi n\theta_1)+\cos(2\pi n\theta_2)),& \text{ if
      } g=2,
    \end{cases}
  \]
  for some real numbers~$\theta_1$ (resp. $\theta_1$ and~$\theta_2$)
  in the interval~$[0,\demi]$.
  
  In the case~$g=2$, we see that~$\Gamma_n(C,\mfm)$ is a Ramanujan graph
  if~$a_C(k_n)\leq 0$, and for~$g=1$, it suffices
  that~$a_C(k_n)\leq -2$. Roughly speaking, this will occur with
  probability $1/2$ or very close to~$1/2$.

  More precisely, assume that~$C$ is ordinary. We then
  have~$\theta_1\not=0$ (resp. $\theta_1\not=0$ and~$\theta_2\not=0$
  if~$g=2$). Consider the set~$X$ in~$(\Rr/\Zz)^g$ of points
  $(\{n\theta_i\})_{1\leq i\leq g}$, where $\{x\}$ denotes the image
  in~$\Rr/\Zz$ of a real number~$x$.  By the Kronecker--Weyl
  equidistribution theorem (see, e.g.,~\cite{pnt}*{Th.\,B.6.5}), the
  set~$X$ is equidistributed as $n\to +\infty$ according to the
  uniform probability measure on its closure~$\bar{X}$, which is a
  closed subgroup of~$(\Rr/\Zz)^g$.
  
  The assumption that~$\theta_i\not=0$ implies straightforwardly that
  \[
    \lim_{N\to +\infty}
    \frac{1}{N}\sum_{1\leq n\leq N}\frac{a_C(k_n)}{\sqrt{|k_n|}}=0,
  \]
  and therefore, by equidistribution, the continuous function
  $f\colon \bar{X}\to \Rr$ defined by
  \[
    f(x)=\sum_{i=1}^g\cos(2\pi x_i)
  \]
  has integral~$0$ on~$\bar{X}$.  Since there are points $(x_i)$
  in~$\bar{X}$ with $f(x)>0$ (e.g. $x=0$), it follows that~$f$ must
  also take negative values on a set of positive measure, which means
  that there exists~$c<0$ such that the measure of the set of
  $(x_i)\in\bar{X}$ with~$f(x)<c$ is positive. By equidistribution,
  there is therefore a positive density set of values of~$n$ such that
  \[
    a_C(k_n)\leq -\frac{c}{2}\sqrt{|k_n|},
  \]
  which is $\leq -2$ if~$|k_n|$ is large enough.

  (5) We now proceed to the proof of Theorem~\ref{it-equi}. We first
  perform a reduction to the case considered in
  Theorem~\ref{th-equi}. Let~$a\in J_{\mfm}(k)$ be the center of the
  symmetric Sidon set $S=i_{\delta}(C\setminus\mfm)(k)$. By assumption,
  there exists $b\in J_{\mfm}(k)$ such that $a=2b$. Let~$S'=-b+S$. Since
  $S=a-S$, we have
  \[
    S'=-b+S=-b+a-S=b-S=-S',
  \]
  i.e., $S'$ is a symmetric Sidon set with center~$0$. In fact, it
  arises as the image of the Abel--Jacobi embedding associated to
  $\delta'=\delta-b$ (viewing the point~$b$ on the generalized jacobian
  as a divisor of degree~$0$). By the last part of
  Proposition~\ref{pr-1}, combined with the fact that
  $J_{\mfm}(\bar{k})$ has at most $2^{2g}\leq 16$ characters of
  order~$2$, the equidistribution property (as $n\to +\infty$) for the
  spectrum of the sum graph associated
  to~$i_{\delta}(C\setminus \mfm)(k_n)$ is equivalent to that
  for~$i_{\delta'}(C\setminus \mfm)(k_n)$. Replacing $\delta$
  with~$\delta'$, we can therefore assume that~$a=0$ and apply
  Theorem~\ref{th-equi}.

  From this theorem, under our assumption that $K^a=K^g=\SU_2(\Cc)$, the
  sums $U_n(\chi)$ are equidistributed as $n\to+\infty$ like the image
  of the Haar measure by the trace map $\SU_2(\Cc)\to [-2,2]$. It is
  classical that this image is none other than the semicircle measure~$\mu_{\mathrm{sc}}$ (\cite{bourbaki}*{p.\,58,\,exemple}).

  Next, by Proposition~\ref{pr-1}, an eigenvalue $\lambda$ in the
  spectrum of $\Gamma_n(C,\mfm)$ may be either
  \[
    \lambda=\frac{|k_n|^{1/2}}{|(C\setminus\mfm)(k_n)|}U_n(\chi)
  \]
  for $\chi$ of order~$2$, or
  \[
    \lambda=\frac{|k_n|^{1/2}}{|(C\setminus\mfm)(k_n)|}|U_n(\chi)|,\quad\text{
      or }\quad
    \lambda=-\frac{|k_n|^{1/2}}{|(C\setminus\mfm)(k_n)|}|U_n(\chi)|,
  \]
  for pairs $(\chi,\bar{\chi})$ of non-real characters. But the
  condition that $i_{\delta}(C\setminus \mfm)$ is a symmetric Sidon
  set with center~$0$ implies that the sums $U_n(\chi)$ are real (and
  hence $S_n(\bar{\chi})=U_n(\chi)$). Thus the number
  $\lambda d_n/\sqrt{d_n-1}$ is of the form
  \[
    \eps \frac{|(C\setminus\mfm)(k_n)|}
    {\sqrt{|k_n|(|(C\setminus\mfm)(k_n)|-1)}}
    |U_n(\chi)|
  \]
  for $\eps=1$ or $\eps=-1$. Since $|k_n|\sim |(C\setminus\mfm)(k_n)|$
  as $n\to +\infty$, and since $U_n(\chi)$ becomes
  $\mu_{\mathrm{sc}}$-equidistributed as seen previously, we obtain
  the equidistribution of $\lambda d_n/\sqrt{d_n-1}$ to
  $\mu_{\mathrm{sc}}$ using the fact that this measure is symmetric
  (i.e., invariant under $x\mapsto -x$).

  (4) Finally, once we know the equidistribution theorem with $K^a=K^g=\SU_2(\Cc)$, we obtain the bound~(\ref{it-w1}) for the
  Wasserstein distance by applying the bound of Kowalski and
  Untrau~\cite{k-u}*{\S\,4.3}.
\end{proof}

Thus, the proof of Theorem~\ref{th-jacobian} will be concluded by the
following result, which establishes the final condition
$K^a=K^g=\SU_2(\Cc)$ in the situations considered in
Theorem~\ref{it-equi}.

\begin{theorem}
  Let~$(C,\mfm)$ be such that $\dim(J_{\mfm})=2$ and $g=1$ or~$2$. The
  group~$K^a$ is infinite, and hence we have $K^a=K^g=\SU_2(\Cc)$, under any
  of the following conditions:
  \begin{enumerate}
  \item[\textup{(i)}] $g=2$.
  \item[\textup{(ii)}] $g=1$, $\mfm$ is a double point and
    $|C(k)|=|k|+1-a_C(k)$ for some integer~$a_C(k)$ which is either~$0$ or has
    multiplicative order at least~$5$ in $k^{\times}$.
  \item[\textup{(iii)}] $g=1$ and the set of unramified characters for
    the object $i_{\delta,*}\bQl[1](1/2)$ on~$J_{\mfm}$ coincides with the
    set of characters that are trivial on the kernel of the 
    projection $J_{\mfm}\to C$.
  \end{enumerate}
\end{theorem}

This will be proved in the next section, specifically in
Proposition~\ref{pr-infinite} for conditions (i) and~(iii), and
Proposition~\ref{pr-infinite2} for condition (ii). 

\begin{remark}
  If we construct a sequence $(\Gamma_m)$ of jacobian graphs
  associated to genus~$2$ curves over prime fields~$\Ff_p$ such that
  $\Gamma_m$ has an odd number~$m$ of vertices, then the equidistribution still holds as
  $m\to+\infty$, in the sense that the numbers
  \[
    \lambda \frac{d_m}{\sqrt{d_m-1}},
  \]
  where $d_m$ is the degree of~$\Gamma_m$ and $\lambda$ ranges over
  the non-trivial eigenvalues of~$\Gamma_m$, become equidistributed
  according to the semicircle law.

  This follows from the ``horizontal'' analogue
  of~\cite{ffk}*{Th.\,4.8} (compare with~\cite{ffk}*{Th.\,4.19}),
  taking into account that the complexity (in the sense of Sawin's
  quantitative sheaf theory~\cite{qst}) of the objects
  $i_{\delta*}\bQl$ on the varying jacobians is uniformly
  bounded. This, in turn, can be deduced from~\cite{qst}*{Prop.\,6.19}
  and the classical fact from algebraic geometry that there exists an
  integer~$N$ such that a curve of genus~$2$ and its jacobian both
  admit projective embeddings in a projective space of
  dimension at most $N$, with the image given by at most~$N$ equations of
  degree at most~$N$.
\end{remark}

\section{Complements on generalized jacobians}

We prove here the statements used in the previous section
concerning some of the properties of the group~$K^a$ in the case of
generalized jacobians associated to a genus~$1$ curve. We keep the
notation from the previous section.

\begin{lemma}\label{lm-swan}
  Assume that the characteristic of~$k$ is $\geq 5$.  Suppose that $C$
  has genus~$1$ and the modulus~$\mfm$ is a double point $2(x_0)$. Let
  $n\geq 1$ be an integer and let $\chi$ be a character of~$J_{\mfm}(k_n)$ which is non-trivial on the kernel of the projection
  $\pi\colon J_{\mfm}\to C$. Then
  \[
    \swan_{x_0}(\mcL_{\chi}|s(C\setminus \{x_0\}))=1. 
  \]
\end{lemma}

\begin{proof}
  We view the curve~$C$ as given by the solutions of a Weierstrass
  equation
  \[
    y^2=x^3+ax+b,
  \]
  together with the point at infinity $\infty$, which is then
  the origin of the group law on~$C$ as an elliptic curve.  Up to translation we may
  assume (possibly up to a finite extension)  that~$x_0\not=\infty$. We have an exact sequence of algebraic groups
  \[
    1\to \Gg_a \to J_{\mfm}\fleche{\pi} C\to 1. 
  \]
    
  The morphism $s\colon C\setminus \{x_0\}\to J_{\mfm}$ that maps $z$
  to the divisor $(z)-(\infty)$ is a partial section of $\pi$. There exists also a partial section
  $\sigma\colon C\setminus \{\infty\} \to J_{\mfm}$ (see, e.g., the
  paper~\cite{cmz}*{\S\,3.2} of Corvaja, Masser and Zannier; this
  reference is written with base field~$\Cc$, but the result only
  depends on the existence of a Weierstrass equation).

  We then define a morphism
  $\Delta\colon C\setminus \{x_0,\infty\}\to J_{\mfm}$ by
  \[
    \Delta(z)=  s(z)-\sigma(z).
  \]

  Since~$s$ and~$\sigma$ are both sections of~$\pi$, the image of
  $\Delta$ is contained in~$\Gg_a$. Thus, we can view~$\Delta$ as a rational function on $C$, with poles at most at $\{x_0,\infty\}$.

  We claim that $\Delta$ has a \emph{simple} pole at $x_0$.  To see
  this, fix some $w_0\in C\setminus \{x_0,\infty\}$ and consider
  the rational function $\kappa$ on~$C$ defined by
  \[
    \kappa(z)=\Delta(z+w_0)-\Delta(z)-\Delta(w_0).
  \]

  It suffices to prove that $\kappa$ has a simple pole
  at~$x_0$. For this we write $\kappa=\kappa_1-\kappa_2$, where
  \begin{gather*}
    \kappa_1(z)=s(z+w_0)-s(z)-s(w_0),\\
    \kappa_2(z)=\sigma(z+w_0)-\sigma(z)-\sigma(w_0).
  \end{gather*}

  The function $\kappa_2$ is regular at~$x_0$. As explained
  in~\cite{cmz}*{\S\,3.5,\,p.\,249}, the rational function~$\kappa_1$
  can be expressed in Weierstrass coordinates~$z=(x,y)$ as a quotient of
  linear forms in~$(x,y)$, where the denominator has a simple pole
  at~$x_0$.

  We can now complete the proof of the lemma. The Swan conductor is a
  local quantity, so 
  \[
    \swan_{x_0}(\mcL_{\chi}|s(C\setminus \{x_0\}))=
    \swan_{x_0}(\mcL_{\chi}|s(C\setminus \{x_0,\infty\})).
  \] Since $\chi$ is a character, there is an isomorphism
  \[
    \mcL_{\chi}|s(C\setminus \{x_0,\infty\})\simeq
    \mcL_{\chi}|\Delta(C\setminus
    \{x_0,\infty\})\otimes\mcL_{\chi}|\sigma(C\setminus
    \{x_0,\infty\}).
  \] Since the function~$\sigma$ is regular at $x_0$, the sheaf
  $\mcL_{\chi}|\sigma(C\setminus \{x_0,\infty\})$ is the restriction
  of the sheaf $\mcL_{\chi}|\sigma(C\setminus \set{\infty})$ which is
  lisse at~$x_0$, and hence
  \[
    \swan_{x_0}(\mcL_{\chi}|s(C\setminus
    \{x_0\}))=\swan_{x_0}(\mcL_{\chi}|\Delta(C\setminus
    \{x_0,\infty\})).
  \] By assumption, the restriction of~$\chi$ to~$\Gg_a$ is a non-trivial
  additive character, hence
  \[
  \swan_{x_0}(\mcL_{\chi}|\Delta(C\setminus \{x_0,\infty\}))
  \] is the
  Swan conductor on~$\Gg_a$ of the sheaf associated to an additive
  character of the function~$\Delta$ with a simple pole at~$x_0$, and it
  is classical that it is equal to~$1$.
\end{proof}

We give two results to prove that the group~$K^a$ is infinite. The first
applies to all curves, although it is conditional when $g=1$.

\begin{proposition}\label{pr-infinite}
  Let $(C,\mfm)$ be data defining a jacobian graph with $g\geq 1$ and
  $\dim J_{\mfm}=2$.  Under the following assumptions, the compact
  group~$K^a$ is infinite.
  \begin{enumth}
  \item If $g=2$ and the curve~$C$ is arbitrary.
  \item If $g=1$ and the set of ramified characters, in the sense
    of~\textup{\cite{ffk}*{Def.\,3.25}}, coincides with the set of
    characters which are trivial on the kernel of the projection
    $\pi\colon J_{\mfm}\to C$.
     \end{enumth}
\end{proposition}

\begin{proof}
  (1) The case $g=2$ is treated in~\cite{ffk}*{Th.\,11.1} (this was
  known to N. Katz around~2012).

  (2) We proceed by contradiction. The assumption that~$K^a$ is finite
  implies that the object 
  \[
  M=i_{\delta*}\bQl[1](1/2)
  \] is generically
  unramified by~\cite{ffk}*{Cor.\,3.39}.  We assume then that the set of
  ramified characters coincides with the set of characters that are
  trivial on the kernel of the projection~$J_{\mfm}\to C$. We denote
  by~$\mcX$ the set of unramified characters.

  Let $n\geq 1$ and let $\chi$ be a character of $J_{\mfm}(k_n)$. From
  the basic theory (see~\cite{ffk}*{\S\,3.9}), if $\chi$ is unramified,
  then the sum
  \[
    U_n(\chi)=\frac{1}{|k_n|^{1/2}}\sum_{x\in (C\setminus
      \mfm)(k_n)}\chi(i_{\delta}(x))
  \]
  is the trace of a matrix belonging to~$K^a$.  If~$K^a$ is finite, it
  follows that this is a sum of roots of unity, and in particular it is
  an algebraic integer (in some cyclotomic field).
 
  Let $y_0\in J_{\mfm}(k_n)$. Using orthogonality of characters, we
  compute
  \[
    \sum_{\chi\in\mcX(k_n)}\overline{\chi(y_0)}U_n(\chi),
  \]
  which is also an algebraic integer. We find
  \begin{align*}
    \sum_{\chi\in \mcX(k_n)}\overline{\chi(y_0)}U_n(\chi) &=
    \frac{1}{|k_n|^{1/2}}\sum_{x\in (C\setminus\mfm)(k_n)}
    \sum_{\chi\in\mcX(k_n)}\overline{\chi(y_0)}\chi(i_{\delta}(x))
    \\
    &= \frac{1}{|k_n|^{1/2}}\Bigl( |(C\setminus\mfm)(k_n)|\delta(y_0)-
    \sum_{\chi\notin\mcX(k_n)}\overline{\chi(y_0)}\chi(i_{\delta}(x)) \Bigr), 
  \end{align*}
  where $\delta(y_0)$ is $1$ if~$y_0$ is in $i_{\delta}(C)(k_n)$ and~$0$
  otherwise. According to our assumption, the condition
  $\chi\notin\mcX(k_n)$ is equivalent to saying that $\chi$ is of the
  form $\chi'\circ \pi$ for some (unique) character~$\chi'$ of the group
  $C(k_n)$. So
  \[
    \sum_{\chi\notin\mcX(k_n)}\overline{\chi(y_0)}\chi(i_{\delta}(x))=
    \sum_{\chi'\in \widehat{C}(k_n)}\overline{\chi'(\pi(y_0))}\chi'(x).
  \]

  This is equal to~$0$ unless $x=\pi(y_0)$, in which case it is
  $|C(k_n)|$. There is a single~$x$ with the second property, so we
  conclude that
  \[
    \frac{1}{|k_n|^{1/2}}\Bigl( |(C\setminus
    \mfm)(k_n)|\delta(y_0)-|C(k_n)| \Bigr)
  \]
  is an algebraic integer.

  Assume that $y_0$ is chosen so that $\delta(y_0)=0$ (such a $y_0$
  exists at least if $n$ is large enough). Write
  $|C(k_n)|=|k_n|+1-a_C(k_n)$ for some integer $a_C(k_n)$.
  Then
  \[
    \sum_{\chi\notin\mcX(k_n)}\overline{\chi(y_0)}\chi(i_{\delta}(x))=
    -\frac{|C(k_n)|}{|k_n|^{1/2}}= -\Bigl(|k_n|^{1/2}
    -\frac{a_C(k_n)-1}{|k_n|^{1/2}}\Bigr)
  \]
  is an algebraic integer, and therefore
  \[
    \beta_n=\frac{a_C(k_n)-1}{|k_n|^{1/2}}
  \]
  is an algebraic integer. By the Hasse bound, we can write further
  $a_C(k_n)=|k_n|^{1/2}(\alpha_n+\alpha_n^{-1})$ with $|\alpha_n|=1$, so
  that
  \[
    \beta_n=\alpha_n+\alpha_n^{-1}-\frac{1}{|k_n|^{1/2}}.
  \]

  All $\beta_n$ are algebraic integers. Now observe that if we replace
  $n$ by $3n$, then the relation
  \[
    \alpha_{3n}+\alpha_{3n}^{-1}=\alpha_n^3+\alpha_n^{-3}
    =(\alpha_n+\alpha_n^{-1})^3-3(\alpha_n+\alpha_n^{-1})
  \]
  holds, and hence
  \[
    \beta_{3n}=\alpha_{3n}+\alpha_{3n}^{-1}-\frac{1}{|k_n|^{3/2}}
  \]
  is congruent modulo algebraic integers to
  \[
    \frac{1}{|k_n|^{3/2}}-\frac{3}{|k_n|^{1/2}}-\frac{1}{|k_n|^{3/2}}
    =-\frac{3}{|k_n|^{1/2}},
  \]
  which is not an algebraic integer (even if~$k$ has
  characteristic~$3$, taking~$n$ large enough). Thus, we have a
  contradiction.
\end{proof}

The assumption concerning ramified characters in part (2) of the
proposition seems quite reasonable. Indeed, a slight elaboration of the
argument used in the computation of $r$ in the beginning of the proof of
Theorem~\ref{th-jacobian} shows that it is \emph{exactly} for the
characters~\hbox{$\chi\in\widehat{J}_{\mfm}(k_n)$} that are non-trivial
on the kernel of $J_{\mfm}\to C$ that the space
$H^0_c(J_{\mfm,\bar{k}},M\otimes \mcL_{\chi})$ is
two\nobreakdash-dimensional and that the action of the Frobenius
automorphism is unitary (i.e., pure of weight~$0$), which are necessary
conditions for~$\chi$ to be unramified.  In fact, it should follow from
the ongoing PhD thesis of Beat Zurbuchen that this condition is also
sufficient, thus proving the result unconditionally.

Independently of these developments, we can prove elementarily that
$K^a$ is infinite in many cases when $g=1$ and~$\mfm$ consists of a
double point. This depends on the following lemma.

\begin{lemma}\label{lm-nonzero}
  Assume that $C$ has genus~$1$ and $\mfm$ is a double point. If $|k|$
  is coprime to $|C(k)|$, then the sum
  \[
    T(\chi)=\sum_{x\in (C\setminus \mfm)(k)}\chi(i_{\delta}(x))
  \]
  is non-zero for all characters~$\chi$ of $J_{\mfm}(k)$.
\end{lemma}

\begin{proof}
  We denote by~$p$ the characteristic of~$k$. We have the exact sequence
  \[
    1\to k\to J_{\mfm}(k)\fleche{\pi} C(k)\to 1.
  \]

  As before, we recall that the size of $C(k)$ is given by the formula
  \[
    |C(k)|=|k|+1-a_C(k),
  \]
  where the integer $a_C(k)$ is of the form
  \[
    a_C(k)=|k|^{1/2}(\alpha+\alpha^{-1})
  \]
  for some algebraic number~$\alpha$ with $|\alpha|=1$.
  
  Since we assume that the order of $C(k)$ is coprime to~$|k|$, the
  above exact sequence splits, and there exists a (section) homomorphism
  $\sigma\colon C(k)\to J_{\mfm}(k)$, defining an isomorphism
  \[
    J_{\mfm}(k)\to k\times C(k)
  \]
  by $x\mapsto (x-\sigma(\pi(x)),\pi(x))$.

  Let~$\chi$ be a character of $J_{\mfm}(k)$, and denote by $\chi_1$
  and~$\chi_2$ the characters of $k$ and $C(k)$ corresponding to $\chi$
  through this isomorphism; we also view these as characters of
  $J_{\mfm}(k)$ (e.g., we write $\chi_2(x)$ for $\chi_2(\pi(x))$ when
  $x\in J_{\mfm}(k)$).

  The values of~$\chi_1$ are $p$-th roots of unity. We write
  \[
    T(\chi)=\sum_{\xi^p=1}\xi\sum_{\substack{x\in (C\setminus
        \mfm)(k)\\\chi_1(i_{\delta}(x))=\xi}}\chi_2(i_{\delta}(x)),
  \]
  where~$\xi$ runs over $p$-th roots of unity in~$\Cc$. For each~$\xi$,
  the inner sum is a cyclotomic integer and belongs to the field
  generated by values of~$\chi_2$, which are $|C(k)|$-th roots of unity.
  Since~$p\nmid |C(k)|$, this field is linearly disjoint from
  $\Qq(e^{2i\pi/p})$, and it follows that $T(\chi)=0$ if and only if the
  cyclotomic integer
  \[
    \sum_{\substack{x\in (C\setminus
        \mfm)(k)\\\chi_1(i_{\delta}(x))=\xi}}\chi_2(i_{\delta}(x)),
  \]
  is \emph{independent} of~$\xi$. If we call this cyclotomic integer
  $\beta$, then we deduce that
  \[
    p\beta=\sum_{\xi^p=1} \sum_{\substack{x\in (C\setminus
        \mfm)(k)\\\chi_1(i_{\delta}(x))=\xi}}\chi_2(i_{\delta}(x))=
    \sum_{x\in (C\setminus
      \mfm)(k)}\chi_2(i_{\delta}(x))=-\chi_2(i_{\delta}(x_0)),
  \]
  where $\mfm=2(x_0)$ (since $i_{\delta}$ is a set-theoretic section
  of~$\pi$ on $C\setminus \mfm$). But this is absurd since the
  right-hand side is a unit whereas the left-hand side is not, and
  therefore $T(\chi)\not=0$.
\end{proof}

\begin{proposition}\label{pr-infinite2}
  Let~$p$ be the characteristic of~$k$.  Suppose that $g=1$ and that
  $\mfm$ is a double point. Write $|C(k)|=|k|+1-a_C(k)$ for some
  integer~$a_C(k)$. If $a_C(k)$ is zero modulo~$p$ or if the
  multiplicative order~$\ord_p(a_C(k))$ of $a_C(k)$ modulo~$p$ is
  $\geq 5$, then the group $K^a$ is infinite. In particular, if~$C$ is
  supersingular, so that $p\mid a_C(k)$, then $K^a$ is infinite.
\end{proposition}

\begin{proof}
  We proceed by contradiction.  As before, the assumption that~$K^a$ is
  finite implies that the object $M=i_{\delta,*}\bQl[1](1/2)$ is
  generically unramified by~\cite{ffk}*{Cor.\,3.39}.

  For any integer~$n\geq 1$ and any unramified character~$\chi$ of
  $J_{\mfm}(k_n)$, the sum $U_n(\chi)$ coincides with the trace of an
  element of~$K^a$. Let~$\gamma$ denote the density in~$K^a$ of the
  elements of trace~$0$. By inspection of the three possible subgroups
  of $\SU_2(\Cc)$ with fourth moment equal to~$2$ (see \cref{it-equid3} and Remark~\ref{rm-data} for details), the
  density~$\gamma$ is at least~$1/4$.
 
  By equidistribution, the proportion of characters~$\chi$ such that the
  sum
  \[
    T_n(\chi)=|k_n|^{1/2}U_n(\chi)= \sum_{x\in (C\setminus
      \mfm)(k_n)}\chi(i_{\delta}(x))
  \]
 vanishes is equal to~$\gamma$ (the proportion is taken in the
  average limit sense of \cref{it-equid1}). We now show that this does not happen under the assumptions of the
  proposition.

  Fix again $n\geq 1$. By Lemma~\ref{lm-nonzero}, applied to $C$
  over~$k_n$ instead of~$k$, none of the sums~$T_n(\chi)$ vanish if
  $|C(k_n)|$ is coprime to~$|k_n|$, or equivalently if
  $p\nmid |C(k_n)|$.
  
  Now we claim that $p\mid |C(k_n)|$ if and only if
  $a_C^n\equiv 1\mods{p}$, where $a_C=a_C(k)$ is the trace of Frobenius
  over the base field~$k$. Indeed, we have 
  \[
    |C(k_n)|\equiv 1-a_C(k_n)\mods{p}. 
  \]
  Writing $a_C=\alpha+\beta$, where
  $\alpha\beta=|k|$, we have $a_C(k_n)=\alpha^n+\beta^n$, and therefore
  \[
    a_C^n=(\alpha+\beta)^n=\alpha^n+\alpha\beta\sum_{j=1}^{n-1}
    \binom{n}{j}\alpha^{j-1}\beta^{n-j-1}+\beta^n\equiv
    \alpha^n+\beta^n\mods{p},
  \]
  which implies that $|C(k_n)|\equiv 1-a_C^n\mods{p}$, proving the
  claim.

  Thus $|C(k_n)|$ is coprime to~$p$ unless $n$ is a multiple
  of~$\ord_p(a_C)$ modulo~$p$. In particular, the
  density~$\gamma$ of characters~$\chi$ with $U_n(\chi)=0$ is
  $\leq 1/{\ord_p(a_C)}$, which contradicts the lower-bound
  $\gamma\geq 1/4$ whenever $\ord_p(a_C)\geq 5$.
\end{proof}

\begin{remark}\label{rm-data}
  Here is the data for the three possible finite groups from
  Theorem~\ref{th-equi}. The results were checked with
  \textsc{Magma}~\cite{magma}, but the groups involved are small and
  well-understood enough that everything can be done by hand using the
  list of conjugacy classes and character tables of the groups involved
  (see, e.g.,~\cite{repr}*{\S\,4.6.4}).

  \begin{enumerate}
  \item $K^a=\SL_2(\Ff_3)$: this is a group of order $24$; it has three
    non-isomorphic two-dimensio\-nal irreducible representations, which
    all have fourth moment equal to~$2$. For each of these, $\Tr(g)=0$
    if and only if~$g$ is conjugate to
    \[
      \begin{pmatrix}
        0 & 1\\-1& 0
      \end{pmatrix}. 
    \]
    This conjugacy class has size~$6=\lvert\SL_2(\Ff_3)\rvert/4$, so the
    density of zero values is $1/4$.
    
  \item $K^a=\GL_2(\Ff_3)$: this is a group of order $48$; it has three
    non-isomorphic two-dimensio\-nal irreducible representations, but only
    two of them have fourth moment equal to~$2$. For both of these,
    $\Tr(g)=0$ if and only if~$g$ is conjugate to one of
    \[
      \begin{pmatrix}
        -1 & 0\\0& 1
      \end{pmatrix},\quad\quad
      \begin{pmatrix}
        -1 & 1\\1& 1
      \end{pmatrix}. 
    \]
    The conjugacy classes have size respectively $12$ and~$6$, so
    that the density of zeros is equal to $18/48>1/4$.

  \item $K^a=\SL_2(\Ff_5)$: this is a group of order $120$; it has two
    non-isomorphic two-dimensio\-nal irreducible representations, both of
    which have fourth moment equal to~$2$. For each of them, we have
    $\Tr(g)=0$ if and only if~$g$ is conjugate to
    \[
      \begin{pmatrix}
        -2 & 0\\0& 2
      \end{pmatrix},
    \]
    whose conjugacy class has size $30=\lvert\SL_2(\Ff_5)\rvert/4$, with a density
    of zeros~$1/4$.
  \end{enumerate}
\end{remark}

\begin{remark}
  (1) The condition $\ord_p(a_C(k))\geq 5$ can be checked efficiently
  using polynomial time algorithms to compute $a_C(k)$. It is also a
  fairly generic condition. Indeed, as already mentioned, it is known by
  work of Deuring and Honda--Tate theory that the integer~$a_C(k)$
  parameterizes the isogeny classes of elliptic curves over~$k$; the
  Hasse bound $|a_C(k)|\leq 2\sqrt{|k|}$ implies that the number of such
  isogeny classes is of order $\sqrt{|k|}$. Since there are only a
  bounded number of elements of order $\leq 4$ modulo~$p$, it follows by
  periodicity that there are only about
  \[
    \frac{4\sqrt{|k|}}{p}+O(1)
  \]
  isogeny classes with $\ord_p(a_C(k))\leq 4$. In particular, if
  $k=\Ff_p$ for some prime~$p$, at most finitely many isogeny classes
  are excluded from about $\sqrt{p}$. The cardinality of each isogeny
  class is a class number of a quadratic order, so that each class
  contains at most about $\sqrt{|k|}$ isomorphism classes. Hence, there
  are only about $4|k|/p+O(\sqrt{|k|})$ isomorphism classes of elliptic
  curves over~$k$ such that $\ord_p(a_C(k))\leq 4$.
  
  (2) If there exists an integer $a$ coprime to~$p$ such that
  $\ord_p(a)\leq 4$, then either $a$ is congruent to $1$ or $-1$
  modulo~$p$ (with order~$1$ or $2$ modulo~$p$), or we must have
  \hbox{$p\equiv 1\mods{3}$} or~\hbox{$p\equiv 1\mods{4}$}. In particular, the only possibility is that $a_C(k)^2\equiv 1\mods{p}$ for
\hbox{$p\equiv 11$} $\mods{12}$; when the base field is~$\Ff_p$, the
  Hasse bound means that, in that case, only \hbox{$a_C(\Ff_p)=1$} or
  $a_C(\Ff_p)=-1$ fail to satisfy $\ord_p(a_C(\Ff_p))\geq 5$.

  The condition $a_C(\Ff_p)\equiv 1\mods{p}$ appears in the arithmetic
  of elliptic curves, in the work of Mazur~\cite{mazur}, who called
  these \emph{anomalous primes}. In particular, it was proved by Mazur
  (see~\cite{mazur}*{proof\,of\,Lemma\,8.18}) that if $C$ is an
  elliptic curve over $\Qq$ with a \emph{non-trivial} rational torsion
  point, then the set of such primes is finite, and in fact is
  contained in the set of primes that may occur as the order of a
  torsion point on an elliptic curve over~$\Qq$, i.e., these are
  $\leq 7$ by another theorem of Mazur~\cite{mazur2}*{Th.\,7'}. Thus,
  in this case, we can apply the proposition to prove that~$K^a$ is
  infinite for any prime $p\equiv 11\mods{12}$.

  (3) We proved elementarily that if~$C$ is ordinary, then the degree of
  the extension of~$k$ generated by the coordinates of $p$-torsion
  points of~$C$ is the same as the multiplicative order of~$a_C(k)$
  modulo~$p$. As pointed out by D. Loeffler, this can also be understood
  more geometrically as follows: the $p$-adic Tate module of~$C$ is, in
  this case, isomorphic to~$\Zz_p$, with the Frobenius action given by
  multiplication by the unique root $\gamma$ of $X^2-a_C(k)X+|k|$, which
  is a $p$-adic unit.  Looking at the $p$-torsion points means looking
  at $\Zz/p\Zz$ with multiplication by~$\gamma$ modulo~$p$, and since
  \[
    X^2-a_C(k)X+|k|\equiv X(X-a_C(k))\mods{p},
  \]
  we have~$\gamma\equiv a_C(k)\mods{p}$.
\end{remark}

Using rather deep results, we can prove furthermore that
Proposition~\ref{pr-infinite2} applies to almost all reductions
modulo primes of a fixed elliptic curve over $\Qq$.
  
\begin{proposition}\label{pr-order}
  Let~$E$ be an elliptic curve over $\Qq$. As $x\to+\infty$, the following holds: 
  \[
    |\{p\leq x\,\mid\, \ord_p(a_E(\Ff_p))\geq 5\}|\sim \frac{x}{\log x}. 
  \]
\end{proposition}

\begin{proof}
  If $\ord_p(a_E(\Ff_p))\leq 2$, then (for $p\geq 7$) we have
  $a_E(\Ff_p)=1$ or $-1$. Applying Serre's $\ell$-adic method for
  instance, one knows that the number of $p\leq x$ with this property is
  $o(\pi(x))$ as $x\to +\infty$ (see~\cite{serre-chebotarev}*{Th.\,20}
  for a more precise estimate; note that this result is stated for
  curves without complex multiplication, but for the case of values $1$
  or $-1$, it extends to the complex multiplication case).

  We can handle the possibility that $\ord_p(a_E(\Ff_p))=3$ or~$4$ in two
  different ways. Both rely on two facts:
  \begin{enumerate}
  \item The condition $\ord_p(a_E(\Ff_p))=3$
    (resp.\ $\ord_p(a_E(\Ff_p))=4$) is equivalent to the fact that
    $x=a_E(\Ff_p)$ is a root modulo~$p$ of a quadratic equation
    $f(x)=0$, namely $x^2+x+1=0$ for order~$3$ and $x^2+1=0$ for
    order~$4$.
  \item Because of the Hasse bound $|a_E(\Ff_p)|\leq 2\sqrt{p}$, this
    root of the equation is ``very small''.
  \end{enumerate}

  For the first, more elementary argument, we deduce that
  $0\leq f(a_E(\Ff_p))\leq 4p+2\sqrt{p}+1$, hence the condition
  $f(a_E(\Ff_p))\equiv 0\mods{p}$ means that
  either
  \[
    a_E(\Ff_p)^2+a_E(\Ff_p)+1=jp,\text{ or }
    a_E(\Ff_p)^2+1=jp
  \]
  for some integer $j$ with $0\leq j\leq 7$. The case $j=0$ is clearly
  excluded. For each value of~$j\geq 1$, an elementary sieve argument
  implies that there are at most $O(\sqrt{x}/\log x)$ primes~$p\leq x$
  such that $jp$ is of the form $n^2+1$ or $n^2+n+1$ (see,
  e.g.,~\cite{hr}*{Th.\,5.4} for the case~$j=1$). Thus, 
  \[
    |\{p\leq x\,\mid\, \ord_p(a_E(\Ff_p))=3\text{ or } 4\}|
    \ll \frac{\sqrt{x}}{\log x}.
  \]
  
  The second, less elementary argument, deduces from the Hasse bound
  that the fractional part $\{\frac{a_E(\Ff_p)}{p}\}$ satisfies
  \[
    0\leq \Bigl\{\frac{a_E(\Ff_p)}{p}\Bigr\}\leq \frac{2}{\sqrt{p}},
  \]
  which cannot happen often because of the result of Duke, Friedlander
  and Iwaniec~\cite{dfi} which shows that the fractional parts
  of the roots of these equations are equidistributed modulo primes.
  
  Precisely, let~$\eps>0$ be given. For $p$ large enough, we get
  \begin{multline*}
    |\{p\leq x\,\mid\, f(a_E(\Ff_p))=0\mods{p}\}| \leq
    \\
    \Bigl|\Bigl\{(p,\nu)\,\mid\, p\leq x,\ \nu\in\Ff_p,\ 0\leq
    \Bigl\{\frac{\nu}{p}\Bigr\}\leq \eps,\ f(\nu)\equiv
    0\mods{p}\Bigr\}\Bigr|,
  \end{multline*}
  and hence, by the Duke--Friedlander--Iwaniec theorem, 
  \[
    \limsup_{x\to+\infty}\frac{1}{\pi(x)}|\{p\leq x\,\mid\,
    f(a_E(\Ff_p))=0\mods{p}\}| \leq \eps. 
  \]
 Since~$\eps$ is arbitrary,
  the result follows again.
\end{proof}

\begin{remark}\
\begin{enumerate}

\item This result suggests a variant of Artin's Primitive Root
  conjecture: given an elliptic curve $E$ over $\Qq$, are there infinitely many
  primes~$p$ such that the Frobenius trace~$a_E(\Ff_p)$ is a primitive root
  modulo~$p$?  This can easily be tested numerically, and the results
  (unsurprisingly) suggest that the answer should be positive, and in
  fact that the density of these primes should also be positive.
  
  Here is a random example: for the curve
  \[
    y^2+y=x^3+x-1,
  \]
  and for primes $p\leq 300\,000$, there are $9\,607$ primes with
  $a_E(\Ff_p)$ a primitive root; for comparison with Artin's Conjecture,
  there are $9\,701$ primes in this range for which $2$ is a primitive
  root. Other random choices of elliptic curves display the same behavior; the
  number of primitive roots is about half of this for complex
  multiplication curves, as it should be since $a_E(\Ff_p)=0$ for about
  half of the primes in this case. (These computations were performed
  with \textsc{Pari/GP}~\cite{parigp}.)

  \item Continuing with an elliptic curve $E$ over $\Qq$, the usual heuristics
  suggest that there should exist infinitely many primes $p$ for which
  $a_E(\Ff_p)$ is of order~$4$, or~$3$, but the number of these $\leq x$
  should grow like $\log\log x$ (because the probability that
  $a_E(\Ff_p)^4=1\mods{p}$ should be about $1/p$, for instance, and
  different primes should behave independently). Numerics for random
  (non-CM) curves are definitely compatible with this, since one finds
  most commonly~$1$ or~$2$ such primes $p\leq 300\,000$ and
  $\log(\log( 300\,000))= 2.53\ldots$.

  \item We do not know how to prove the statement of
  Proposition~\ref{pr-order} if the condition $\ord_p(a_E(\Ff_p))\geq 5$
  is replaced by $\ord_p(a_E(\Ff_p))\geq 6$. Indeed, $a_E(\Ff_p)$
  satisfies then an equation of degree~$\geq 4$, and there are too many
  possibilities to apply the elementary sieve argument. And although a
  solution of (say) $\ord_p(a_E(\Ff_p))=5$ gives a small solution of the
  equation $X^4+X^3+X^2+X+1=0$ modulo~$p$, the equidistribution of
  fractional parts of roots of this equation modulo primes is not known
  (indeed, \emph{no} irreducible polynomial of degree~$\geq 3$ is
  currently known to satisfy the equidistribution property, although it
  is conjectured to hold).

  However, J. Merikoski pointed out that, for any integer~$m\geq 6$,
  one can still prove that there is a positive proportion of
  primes~$p$ for which $\ord_p(a_E(\Ff_p))$ is $\geq m$. Indeed, by
  Dirichlet's theorem on primes in arithmetic progressions, there is a
  positive proportion of primes $p$ such that $p-1$ is not divisible
  by any odd prime $<m$ (e.g., primes congruent to~$2$ modulo the
  product of the odd primes $<m$). For such primes~$p$, the order of
  $a_E(\Ff_p)$ modulo~$p$ is either~$\leq 2$ or~$>m$. By the argument
  at the beginning of the proof of Proposition~\ref{pr-order}, the
  number of primes~$p\leq x$ with $\ord_p(a_E(\Ff_p))\leq 2$ is
  $o(\pi(x))$, so we obtain the result.  
  \end{enumerate}
\end{remark}

\appendix
\section{Non-technical overview of generalized jacobians}
\label{app-gj}

We give here a quick intuitive overview of the definition of generalized
jacobians for non\nobreakdash-experts, assuming no background in algebraic geometry
(an accessible account can be found, for instance,
in~\cite{silverman}*{Ch.\,1,\,2}). The lecture~\cite{renyi-talk} might
also be helpful, as well as the more concise description of the
construction in~\cite{ffk2}*{\S\,2}.

Let $\Ff$ be a perfect field,  and~$\bar{\Ff}$ an algebraically closed field
containing~$\Ff$. An \emph{affine plane algebraic curve} $C$ over $\Ff$ is defined as the zero set in~$\bar{\Ff}^2$ of a non-constant two-variable
polynomial $f\in \Ff[X,Y]$, and the \emph{$\Ff$-rational points} on the
curve are those points $(x,y) \in \Ff^2$ such that $f(x,y)=0$;
the set of those points is denoted by~$C(\Ff)$.

The curve~$C$ is said to be \emph{irreducible} if~$f$ is irreducible,
and \emph{non-singular} (or smooth) if, for every point
$(x,y)\in\bar{\Ff}^2$ on~$C$, one of the partial derivatives
$\partial_xf(x,y)$ or~$\partial_yf(x,y)$ is non-zero. An important
special case is when $f$ is of the form
\begin{equation}\label{eq-hyp}
  f(X,Y)=Y^2-g(X),\quad\quad g\in \Ff[X],
\end{equation}
and the condition is then that~$g$ has simple roots
in~$\bar{\Ff}$. (These curves are called \emph{hyperelliptic}.)

Given an irreducible non-singular curve~$C$ and a multiset $\mfm$ of
rational points on~$C$ (usually called a \emph{modulus}, or an
\emph{effective divisor} on~$C$), there is an associated group called
the \emph{generalized jacobian}, and denoted by~$J_\mfm(C)$. Additionally,
except in very special cases, one shows that for a choice of a
point~$p_0=(x,y)$ on the curve, there is an associated injective map
$i\colon C\setminus\mfm \to J_\mfm(C)$, called an \emph{Abel--Jacobi
  map}; it gives us a natural way of viewing~$C\setminus \mfm$ itself as
a subset of $J_\mfm(C)$. For instance, this holds for hyperelliptic
curves if the degree of the polynomial~$g$ in~(\ref{eq-hyp}) is
greater than or equal to $3$.

Quite surprisingly, it was proved in~\cite{ffk2} that, under mild
assumptions on the curve and the modulus, the image by~$i$ of
$C\setminus \mfm$ always yields a Sidon or symmetric Sidon subset of the
group $J_\mfm(C)$. These are infinite sets in infinite groups, but if
one starts with a finite field~$\Ff$, one can take the subset of
elements invariant under the action of the \emph{Frobenius automorphism}
(in the case of~$C$, this means elements with $(x,y)\in\Ff^2$) to obtain
finite sets and finite groups. This is the starting point of our work.

We will now define more precisely the generalized
jacobian~$J_\mfm(C)$ and the Abel--Jacobi maps $i$, in two
concrete special cases, both of which are particularly relevant to our
main results: elliptic curves with $\mathfrak m$ consisting of two
distinct points, and curves of genus~$2$ with~$\mathfrak
m=\emptyset$. In the spirit of accessibility, we describe these two examples in different languages: for the first, we give a completely explicit description of the group $J_\mfm(C)$, whereas for the second we use more abstract language, which
extends readily to a complete definition of all generalized jacobians.

\textbf{Elliptic curves.}  We begin with the case of elliptic
curves. For our purposes,\footnote{\,As long as the characteristic of
  $\Ff$ is not $2$ or $3$, every elliptic curve can be put into this
  form, called \emph{Weierstrass form}.} an elliptic curve $E$ is the set of solutions to the polynomial
equation $y^2=x^3+ax+b$, for some constants~$a,b \in \Ff$, chosen so
that the polynomial $x^3+ax+b$ has simple roots in~$\bar{\Ff}$. It
is not hard to see that if a line $\ell$ passes through two points
$P,Q$ of $E$, then it also passes through a third point; this even
holds if $P=Q$, so long as we interpret $\ell$ as the tangent line to $E$ at $P$.\footnote{\,The assumption that
  the polynomial has simple roots is precisely what is needed to
  ensure the existence of this tangent line at any point of the
  curve.} Precisely, the only exceptions to this rule are vertical
lines, which pass through only two points of $E$, but this is
corrected by adding to the solutions of the equation a ``point at
infinity'', denoted~$\infty$, which lies on all vertical lines.
Given a point $P=(x,y)$, we denote by~$-P=(x,-y)$ the point obtained
by reflecting about the $x$-axis; note that $-P$ also lies on~$E$, and
that $P, -P$, and $\infty$ are collinear. We denote by $E(\bar{\Ff})$
the set of points on~$E$, including $\infty$, and by~$E(\Ff)$ the
subset of those with coordinates in~$\Ff$.

The fact above allows us to define a binary operation on
$E(\bar{\Ff})$, by declaring $P*Q$ to be the unique third point lying
on the line between $P$ and $Q$. We then define $P+Q$ to be equal
to~$-(P*Q)$; that is, $P+Q$ is computed by taking the line through $P$
and $Q$, finding its unique third intersection with $E$, and
reflecting that point about the $x$-axis. Surprisingly, this rule
defines an abelian group on the set $E(\bar{\Ff})$ of points of $E$,
in which the identity element is $\infty$ and the additive inverse of
$P$ is $-P$. Moreover, it is elementary that~$E(\Ff)$ is a subgroup of~$E(\bar{\Ff})$, which is finite if~$\Ff$ is finite.

Given the group law on $E(\Ff)$, it is not hard to define the
generalized jacobian. In what follows, we rely on work of Déchène
\cite{dechene}*{\S\,5.3}, who explicitly worked out what the abstract
generalized jacobian corresponds to in this special case. Let
$M,N \in E(\Ff)$ be two distinct points on $E$ (with coordinates
in~$\Ff$), both distinct from~$\infty$, and let $\mfm$ consist of~$M$
and~$N$, both with multiplicity~$1$. The set underlying the
generalized jacobian $J_\mfm(E)$ is
$E(\bar{\Ff})\times \bar{\Ff}^\times$, i.e., the elements of the group
$J_\mfm(E)$ are pairs $(P,s)$, where $P$ is a point on $E$ and
$s \in \bar{\Ff}$ is non-zero. The group operation is not the ``obvious''
componentwise operation, but is defined (generically) by
\begin{equation}\label{eq:jacobian law}
  (P,s) \cdot (Q,t) = \left( P+Q, s\cdot t \cdot
    \frac{\ell_{P,Q}(M)}{\ell_{P+Q,\infty}(M)}\cdot
    \frac{\ell_{P+Q,\infty}(N)}{\ell_{P,Q}(N)} \right),
\end{equation}
where $\ell_{P_1,P_2}$ denotes the equation of the line passing
through the points $P_1,P_2$. That is, in the first coordinate we
simply apply the group law on $E(\bar{\Ff})$ described above, whereas
the second coordinate multiplies the elements of $\Ff^\times$ together
with an additional quantity involving the points $P,Q$ as
well as the fixed modulus $\mfm=\{M,N\}$\footnote{\,As written, this
  group law depends on the choice of which point is $M$ and which is
  $N$, rather than on the unordered pair, but it is not hard to see
  that one obtains an isomorphic group by swapping their roles. More
  problematically, the expression in \eqref{eq:jacobian law} does not
  make sense if $\{P,Q,P+Q\} \cap \{M,N\} \neq \emptyset$, as we are
  then dividing by zero; to define the group operation in this case,
  one uses an identical formula, simply shifted by some
  \emph{translation point}; see \cite{dechene}*{Prop.\,5.5} for
  details.}. The identity element of the group is the pair
$(\infty,1)$.  Finally, the Abel--Jacobi map
$i$ is defined  by $i(P) = (P,1)$. The
subgroup of $\Ff$-rational points is simply the subgroup where
$P\in E(\Ff)$ and~$s\in\Ff^{\times}$.

We stress that this description of the group law of $J_\mfm(E)$ is
completely explicit, also in the sense that all expressions arising in it can be
efficiently computed.

As an illustration,  let us quickly verify that this construction does
indeed give a symmetric Sidon set, following the argument of
\cite{ffk2}*{Lemma 7}. As previously stated, this is an infinite
subset of~$J_{\mfm}(C)$, but taking the points with coordinates
in~$\Ff$ gives finite sets and groups when~$\Ff$ is finite.

\begin{proposition}
  Let $E$ be an elliptic curve over $\Ff$, and let $\mfm = \{M,N\}$
  for some distinct points $M,N \in E(\Ff)$, both distinct
  from~$\infty$. Then the set $S = \{(P,1) : P \in E(\bar{\Ff})\}$ is
  a symmetric Sidon subset of $J_\mfm(E)$, with center
  $(-(M+N),1) \in S$.
\end{proposition}

\begin{proof}
  Suppose that $P_1,Q_1,P_2,Q_2 \in E(\Ff)$ satisfy
  $(P_1,1)\cdot (Q_1,1) = (P_2,1)\cdot (Q_2,1)$; this means, by the
  group law \eqref{eq:jacobian law}, that
  \[
    P_1+Q_1 = P_2+Q_2
  \]
  and
  \[
    \frac{\ell_{P_1,Q_1}(M)}{\ell_{P_1+Q_1,\infty}(M)}\cdot
    \frac{\ell_{P_1+Q_1,\infty}(N)}{\ell_{P_1,Q_1}(N)}=
    \frac{\ell_{P_2,Q_2}(M)}{\ell_{P_2+Q_2,\infty}(M)}\cdot
    \frac{\ell_{P_2+Q_2,\infty}(N)}{\ell_{P_2,Q_2}(N)}.
  \]
  The second equation can be simplified since $P_1+Q_1=P_2+Q_2$,
  becoming
  \[
    \frac{\ell_{P_1,Q_1}(M)}{\ell_{P_1,Q_1}(N)} =
    \frac{\ell_{P_2,Q_2}(M)}{\ell_{P_2,Q_2}(N)}
    \qquad\Longleftrightarrow\qquad
    \frac{\ell_{P_1,Q_1}(M)}{\ell_{P_2,Q_2}(M)} =
    \frac{\ell_{P_1,Q_1}(N)}{\ell_{P_2,Q_2}(N)}.
  \]
  
  Let $R=-(P_1+Q_1)=-(P_2+Q_2)$; this is the intersection point of the
  lines through $P_1,Q_1$ and through $P_2,Q_2$. We are done if
  $\{P_1,Q_1\}=\{P_2,Q_2\}$, so we may assume this does not happen;
  the key observation in this case is that the ratio
  $\ell_{P_1,Q_1}(M)/\ell_{P_2,Q_2}(M)$ determines and is determined
  by the slope of the line through $M$ and $R$. As a consequence, we
  find that the slope of the line through $M$ and $R$ equals the slope
  of the line through $N$ and $R$, and hence the points $M,N,R$ are
  collinear. By the definition of the group law on the elliptic curve,
  we find that $R=-(M+N)$. But this shows that $R$ is independent of
  $P_1,P_2,Q_1,Q_2$; that is, whenever
  $(P_1,1)\cdot (Q_1,1)=(P_2,1)\cdot (Q_2,1)$, we either have that
  $\{P_1,Q_1\}=\{P_2,Q_2\}$, or that
  $(P_1,1)\cdot (Q_1,1)=(-(M+N),1)=(P_2,1)\cdot (Q_2,1)$. As it is
  straightforward to verify that $S$ is symmetric about $(-(M+N),1)$,
  we conclude that $S$ is a symmetric Sidon set.
\end{proof}

\textbf{Genus~$2$ curves.} We now turn to the second example, which
concerns curves of genus~$2$. For us, such a curve over $\Ff$ is an
algebraic curve given by an equation $y^2=f(x)$, where $f$ is a
polynomial of degree $5$ or $6$ with no repeated roots in
$\bar{\Ff}$. In this case, we only deal with the case where the
modulus $\mfm$ is empty; in such instances, the generalized jacobian
is nothing more than the ``classical'' jacobian $J(C)$ of $C$, which
we now define. As in the previous case, we need to add to~$C$ some
points at infinity to obtain a smooth projective curve; there is only
one if~$\deg(f)=5$ and there are two if~$\deg(f)=6$. The
notation~$C(\bar{\Ff})$ or~$C(\Ff)$ always mean the sets of points
including these.

The group $J(C)$ is defined as the quotient of two infinite abelian
groups. The first is the group of \emph{degree-zero divisors}
on~$C$. Here, a \emph{divisor} is nothing more than a formal linear
combination, with integer coefficients, of points on $C$. The divisors
form a group, denoted by~$\Div(C)$, under the addition of formal linear
combinations. In other words, $\Div(C)$ is simply the free abelian
group generated by the points of $C(\bar{\Ff})$. It contains the
subgroup $\Div^0(C)$ of degree-zero divisors, namely those divisors
whose coefficients sum to zero.

A special class of divisors are \emph{principal divisors}, which are
those arising from non-zero rational function, that is, from ratios
$r(X,Y)=p(X,Y)/q(X,Y)$ of non-zero polynomials
in~$\bar{\Ff}[X,Y]$. Namely, each such rational function gives us a
divisor $\divv(r)$ recording the roots and poles of $r$. More
precisely, $\divv(r)$ is defined to be the sum of the points $P$ at
which \hbox{$r(P)=0$} (counted with multiplicity), minus the sum of the
points $P$ at which \hbox{$r(P)=\infty$} (again with multiplicity). A basic
fact of algebraic geometry is that if $\deg p=\deg q$, then
$\divv(p/q)$ is a degree-zero divisor; this is essentially a
restatement of the fact that a degree-$d$ polynomial has exactly $d$
roots, counted with multiplicity, over every algebraically closed
field. Since $\divv(rr') = \divv(r)+\divv(r')$, we see that the set of
principal divisors forms a subgroup $P(C)$ of $\Div^0(C)$. The
\emph{jacobian} is then defined as the quotient $J(C)=\Div^0(C)/P(C)$.

If~$\Ff$ is finite, then the group $J(C)(\Ff)$ of $\Ff$-rational
points of~$J(C)$ (which is a finite subgroup of~$J(C)$) is defined as
the subgroup of classes of divisors which are invariant under the
Frobenius automorphism~$\sigma$ acting on divisors by the rule
\[
  \sum_{P} n_P (P)\mapsto \sum_P n_P (\sigma(P)).
\]

\begin{remark}
  Being invariant does not necessarily mean that the points~$P$ for
  which $n_P\not=0$ are in~$C(\Ff)$: for instance, if~$P=(x,y)$ with
  $x$ and~$y$ in the quadratic extension of~$\Ff$ and $Q$ is another point of $C(\Ff)$, then
  $(P)+(\sigma(P))-2(Q)$ is in $J(C)(\Ff)$.
\end{remark}

As before, the Abel--Jacobi map is essentially the most natural way
one would guess to embed the points of $C$ into $J(C)$. For instance, if one of the points at infinity $\infty$ is rational, then every point $P \in C(\Ff)$
yields a degree-zero divisor $(P)-(\infty)$, which yields an embedding of
the points of $C$ into $\Div^0(C)$. By composing with the quotient map
$\Div^0(C) \to J(C)$, we obtain the Abel--Jacobi embedding $C(\Ff) \to
J(C)$.

Again, it turns out that the image of this map is a symmetric Sidon
subset of the group~$J(C)$. This is not very hard to prove, but
requires some background theory on hyperelliptic curves, so we refer
to~\cite{ffk2} for details.

Although our description is abstract, there are
efficient algorithms to represent points of~$J(C)(\Ff)$ and 
operate on them (see, for instance, the recent description by Flynn and
Khuri-Makdisi~\cite{f-km} of $J(C)$ as a subvariety of
$\mathbf{P}^3\times \mathbf{P}^3$ with an explicit form of the
addition~map).

\bibliographystyle{Nabbrv}
\bibliography{jacobian-graphs}

\end{document}